\documentclass[final,leqno]{siamltex704}
\usepackage{hyperref}
\usepackage{amsmath}
\usepackage{graphicx}
\usepackage{array}
\usepackage[notcite,notref]{showkeys}
\usepackage{tikz}
\usepackage{threeparttable}

\usepackage{amsfonts,amssymb}
\usepackage{dsfont}
\usepackage{pifont}
\usepackage{subfigure}

\def\argmin{\operatornamewithlimits{arg\ min}}
\newtheorem{remark}{Remark}
\newtheorem{algorithm}{Algorithm}
\newtheorem{Defi}{Definition}

\def\NN{\hbox{\rlap{I}\kern.16em N}}
\def\NC{\hbox{\rlap{\kern.24em\raise.1ex\hbox
                  {\vrule height1.3ex width.9pt}}C}}

\def\LL{{\cal{L}}}

\def\pT{{\partial T}}
\def\bR{\mathbb{R}}

\begin{document}

\setlength{\parindent}{0.25in} \setlength{\parskip}{0.08in}

\title{An $L^p$-Weak Galerkin Method for second order elliptic equations in non-divergence form}

\author{
Waixiang Cao\thanks{School of Mathematical Sciences, Beijing Normal University,  Beijing, 100875, China. The research of W. Cao was supported in part by National Natural Science Foundation of China grant  11871106.}
\and
Junping Wang\thanks{Division of Mathematical
Sciences, National Science Foundation, Alexandria, VA 22314
(jwang@nsf.gov). The research of J. Wang was supported in part by the
NSF IR/D program, while working at National Science Foundation.
However, any opinion, finding, and conclusions or recommendations
expressed in this material are those of the author and do not
necessarily reflect the views of the National Science Foundation.}
\and Yuesheng Xu\thanks{Department of Mathematics and Statistics, Old Dominion University, Norfolk VA 23529 (y1xu@odu.edu). The research of Y. Xu was supported in part by the NSF under
Grant DMS-1912958.}
}

\maketitle

\begin{abstract}
This article presents a new primal-dual weak Galerkin method for second order elliptic equations in
non-divergence form. The new method is devised as a constrained $L^p$-optimization problem with constraints that mimic the second order elliptic equation by using the discrete weak Hessian locally on each element. An equivalent min-max characterization is derived to show the existence and uniqueness of the numerical solution. Optimal order error estimates are established for the numerical solution under the discrete $W^{2,p}$ norm,  as well as the standard $W^{1,p}$ and $L^p$ norms. An equivalent characterization of the optimization problem in term of a system of fixed-point equations via the proximity operator is presented. An iterative algorithm is designed based on the fixed-point equations to solve the optimization problems. Implementation of the iterative algorithm is studied and convergence of the iterative algorithm is established. Numerical experiments for both smooth and non-smooth coefficients problems are presented to verify the theoretical findings.
\end{abstract}



\section{Introduction}
This paper is concerned with the development of new numerical methods for second-order elliptic problems in non-divergence form. For simplicity, we consider the model problem that seeks an unknown function $u=u(x)$ such that 
\begin{eqnarray}\label{poission}
\begin{aligned}
   {\cal L}u:=\, \sum_{i,j=1}^da_{ij} \partial_{ij}^2 u&=f,  &&  {\text in} \ \Omega, \\
    u&=0,\ \  && {\text on} \ \partial\Omega,
\end{aligned}
\end{eqnarray}
where  $\Omega$ is a polygonal or polyhedral domain, $a_{ij}\in L^\infty(\Omega)$, and $f\in L^p(\Omega)$ is a given function with $p\in [1,\infty)$. Assume that  the coefficient matrix $a(x)=(a_{ij}(x))_{d\times d}$, $x\in \Omega$, is
symmetric, uniformly bounded and positive definite; i.e., there exist positive constants $\alpha$ and $\beta$ such that
 \[
     \alpha\xi^T\xi\le \xi^Ta(x)\xi\le  \beta\xi^T\xi,\ \ \forall \xi\in {\mathbb R}^d, \qquad \forall x\in\Omega.
\]

Equations in non-divergence form, such as \eqref{poission},  have attracted much attention recently
due to its wide application in many science and engineering problems such as the probability and stochastic process \cite{Fleming-Soner}, the linearized fully nonlinear problems \cite{Brenner-Sung,Neilan}, and the fully nonlinear Hamilton-Jacobi-Bellman equations (see, e.g., \cite{Jensen-Smears}).
The features of low regularity on the coefficient matrix and the non-divergence structure make the problem \eqref{poission} hard to admit any traditional variational formulation amenable to the use of standard Galerkin finite element techniques in numerical approximation.  Recently, several numerical methods based on Galerkin type methods have been developed (see, e.g., \cite{Feng-Neilan,Lakkis,Nochetto,Smears}). In \cite{Lakkis},  Lakkis and Pryer proposed a new mixed finite element method for solving the problem \eqref{poission} by constructing  an appropriate {\it finite element Hessian}. Later, Neilan \cite{Neilan2017} modified  the definition of the finite element Hessian rendering the auxiliary variable completely local, thus resulting in a more efficient scheme.  In addition,  error estimates in a discrete $H^2$-norm were derived. In \cite{WW-mathcomp}, a primal-dual weak Galerkin method was introduced to approximate the solution of the problem \eqref{poission}. The basic idea of the primal-dual weak Galerkin is to use the weak Galerkin strategy (see, e.g., \cite{mu-wang-ye-2015,Wang:YeWG2013,wang-ye-2014,wang-ye-2015})  to construct a {\it discrete weak Hessian} operator, and then seek a discontinuous function which satisfies the given PDE
\eqref{poission} weakly on each element (i.e., replace the operator ${\cal L}$ by  the specially designed discrete weak Hessian operator), plus minimizing the $L^2$ stabilizer on the boundary of each element. The constrained optimization problem gives rise to a symmetric linear system involving not only the primal variable $u$ but also a dual variable (say $\lambda$) known as the Lagrangian multiplier. Optimal order error estimates were derived for the finite element approximation in the discrete $H^2$, as well as the standard $H^1$ and
$L^2$-norms.

The goal of this paper is to present a new primal-dual weak Galerkin method by using an $L^p$ stabilizer for the problem \eqref{poission}, and then establish a general $L^p$ theory for the corresponding numerical method.  To our best knowledge, there are no existing results of $L^p$ theory for weak Galerkin finite element methods in the literature, and this paper is the first along the direction of $L^p$.
Different from the method in \cite{WW-mathcomp}, here our numerical scheme is formulated as a constrained $L^p$ optimization problem with constraints that satisfy the PDE \eqref{poission} weakly on each element. To study the convergence behavior of the numerical solutions arising from the constrained $L^p$ optimization, we shall present an equivalent Euler-Lagrange form and a min-max characterization for the numerical scheme.  Based on the assumption that the problem \eqref{poission} has the $W^{2,p}$ regularity (that is, the solution to  $\LL \varphi = f,\ \varphi|_{\partial\Omega} = 0$ is $W^{2,p}$ regular and satisfies
$
\|\varphi\|_{2,p} \leq C \|f\|_{0,p}
$ ),
we  establish optimal order error estimates  for the finite element approximation in the discrete $W^{2,p}$, as well as the standard $W^{1,p}$- and $L^p$-norms. 
In addition,  an iterative algorithm and its detailed implementation are developed for solving the constrained nonlinear $L^p$ optimization problem numerically.
In particular, when $p=1$ we need to solve a minimization problem with non-differentiable objective functions, for which standard optimization methods are not applicable. Inspired by recent development in image science \cite{krol-li-shen-xu, li-shen-xu-zhang, li-Micchelli-sheng-xu, Micchelli-Shen-Xu}, we design a fixed-point iterative algorithm via the proximity operator of the non-differentiable functions that appear in its objective function. We establish a convergence result for the iterative algorithm and study crucial issues in implementation of the algorithm.

We would like to emphasize that the current  work is our first attempt to relate the new primal-dual weak Galerkin method to a constrained $L^p$ optimization problem with $p\neq 2$,  where error estimates, iterative algorithm and its convergence and implementation are of major interests and study. Analysis of  other important advantages and features of the $L^p$ optimization problem (e.g., using $p=1$ and wavelet basis approximation may yield sparse numerical solutions) is equally interesting and challenging,  and thus deserves a separate investigation.

The rest of the paper is organized as follows. In Section 2, we review some preliminary results and introduce basic notations. In Section 3, we present the numerical scheme based on constrained $L^p$ optimizations.  In Section 4, we shall rewrite the numerical scheme into
an equivalent min-max characterization. Section 5 is devoted to a discussion of the solution existence and uniquness for the constrained $L^p$ optimization problem. Section 6 is the most technical and main body of this paper, where optimal order error estimates are established for the primal variable in the discrete $W^{2,p}$, $W^{1,p}$, and $L^p$ norms. Section 6 additionally derives an $L^q$-estimate for the Lagrangian multiplier.  In Section 7, we present a characterization of the optimization problem in term of a system of fixed-point equations via proximity operator and design an iterative algorithm to solve the nonlinear optimization problems numerically. The convergence of the iterative algorithm is studied in Section 8.  In Section 9, we demonstrate how to implement the iterative algorithm and calculate the proximity operator. Finally,  some numerical examples are presented in Section 10 to support our theoretical findings.

\section{Notations and Preliminaries}
We adopt the standard notation for Sobolev spaces such as $W^{m,p}(D)$ on subdomain $D\subset\Omega$ equipped with the norm $\|\cdot\|_{m,p,D}$ and the semi-norm $|\cdot|_{m,p,D}$. When $D=\Omega$, we shall omit the index $D$ in the norm notation. For $p=2$, we set
$H^m(D)=W^{m,p}(D)$, $\|\cdot\|_{m,D}=\|\cdot\|_{m,p,D}$, and $|\cdot|_{m,D}=|\cdot|_{m,p,D}$.

\begin{Defi}
By a classical solution  of \eqref{poission}, we mean a function $u\in W^{2,p}(\Omega)\cap W_0^{1,p}(\Omega)$ satisfying ${\cal L}u=f$ a.e. in $\Omega$.
\end{Defi}

Throughtout this paper, we assume the elliptic operator
${\cal L}=\sum_{i,j=1}^da_{ij}\partial_{ij}^2$, together with the usual Dirichlet boundary condition, have the $W^{2,p}$-regularity in the sense that the solution to
\[
    {\cal L}\psi=f,\ \ \psi |_{\partial\Omega}=0
\]
  is $W^{2,p}$-regular and satisfies the  following a priori estimate
 \[
    \|\psi\|_{2,p}\le C \|f\|_{0,p}.
 \]
  Consequently, the following {\it inf-sup} condition holds true
 \[
     \sup\limits_{\psi\in W^{2,p}\cap W^{1,p}_0}\frac{|({\cal L}\psi,w)|}{\|\psi\|_{2,p}}\ge \beta \|w\|_{0,q},
 \]
where $q=p/(p-1)$ is the conjugate of $p$.
 
Let ${\cal T}_h$ be a partition of the domain $\Omega$, and denote by ${\cal E}_h$ the set of all edges in ${\cal T}_h$, and ${\cal E}^0_h$ the set of all interior edges.
For each element $T\in {\cal T}_h$, we denote by $h_{T}$ the diameter and the mesh size $h=\max_{T \in {\cal T}_h} h_{T}$ for ${\cal T}_h$.

 Let $K$ be a polygonal or polyhedral domain with boundary $\partial K$. By a weak function on $K$, we mean a triplet $v=\{v_0,v_b, {\bf v}_g\}$ such that
 $v_0\in L^2(K), v_b\in L^2(\partial K)$ and  ${\bf v}_g=(v_{g1},\ldots, v_{gd})\in [L^2(\partial K)]^d$. Here $v_0$ and $v_b$ can be understood as the value of $v$ in the interior
 and on the boundary, while ${\bf v}_g$ represents the gradient $\nabla v$ on the boundary. Note that $v_0$ and ${\bf v}_g$ are not necessarily the trace of $v_0$ and $\nabla v_0$ on $\partial K$, respectively. Denote by
 \[
    W(K):=\{v=\{v_0,v_b, {\bf v}_g\}: \ v_0\in L^2(K), v_b\in L^2(\partial K), {\bf v}_g\in [L^2(\partial K)]^d\}
 \]
the space of all weak functions on $K$. Recall that weak Hessian is a $d\times d$ matrix consisting of weak second order partial derivatives $\partial_{ij,w}^2v$ defined as a linear functional on the Sobolev space $H^2(K)$ so that its action on $\psi\in H^2(K)$ is given by
 \begin{equation}
   (\partial_{ij,w}^2v,\psi)_K:=(v_0,\partial_{ji}^2\psi)_K-\langle v_bn_i, \partial_j\psi\rangle_{\partial K}+\langle v_{gi},\psi n_j\rangle_{\partial K}.
 \end{equation}
   Here  ${\bf n}=(n_1,\ldots, n_d)$ denotes the unit outward normal direction on $\partial K$, and
    for any function $v,\psi$,
\[
 (v,\psi)_K:=\int_{K} v\psi,\ \  \langle v,\psi\rangle_{\partial K}:=\int_{\partial K} v\psi ds.
\]
It follows that the weak Hessian of $v\in W(K)$ is defined as follows:
 \[
      \nabla_{w} v:=(\partial^2_{ij,w}v)_{i,j=1}^d.
 \]
  
Analogously, we can define the weak second order partial derivatives and weak Hessian on any finite dimensional linear spaces. Specifically, for a given finite dimensional space $S_r(K)$ on $K$, the weak second order partial derivative  on $S_r(K)$ is defined as a function  $\partial_{ij,w}^2v\in S_r(K)$ satisfying
\begin{equation}\label{weak:deri}
   (\partial_{ij,w}^2v,\psi)_K=(v_0,\partial_{ji}^2\psi)_K-\langle v_bn_i, \partial_j\psi\rangle_{\partial K}+\langle v_{gi},\psi n_j\rangle_{\partial K},\ \ \forall \psi\in S_r(K).
 \end{equation}

\section{Numerical Scheme}
For any given integer $k\ge 2$, we denote by $V_h$ the subspace of $W(K)$ consisting of piecewise polynomials in the following form
\begin{equation}\label{Vh}
   V_h:=\{v=\{v_0,v_b, {\bf v}_g\}:\ v|_{T}\in P_k(T)\times P_k(e)\times [P_{k-1}(e)]^d,\ e\in \partial T, T\in {\cal T}_h \}.
\end{equation}
  Denote by $V_h^0$ the subspace of $V_h$ with vanishing boundary value for $v_b$ on the boundary $\partial\Omega$. That is,
\begin{equation}\label{vh0}
    V^0_h:=\{v=\{v_0,v_b, {\bf v}_g\}\in V_h:  v_b|_{e}=0,\ \ e\in\partial\Omega \}.
\end{equation}
Let $W_h$ be a linear space of polynomials satisfying
  \begin{equation}
     W_h:=\{\sigma: \sigma|_{T}\in P_{l}, l=k-2,  \ {\rm or} \ k-1,\ \  T \in{\cal T}_h\}.
  \end{equation}
 
Let the weak differential operator ${\cal L}_w$ be given by
$$
{\cal L}_w: = \sum_{i,j=1}^d a_{ij} \partial_{ij,w}^2,
$$
where $\partial_{ij,w}^2$ is the weak second order partial derivative operator defined in \eqref{weak:deri} with $S_r=W_h$.
Introduce the following $L^p$ stabilizer for any $p\ge 1$:
 \begin{eqnarray*}
   s(v):&=&\left\{
   \begin{array}{ll}\displaystyle
   \frac{1}{p} \sum_{T\in{\cal T}_h}\int_{\partial T}\left(h_T^{1-2p}|v_0-v_b|^p+h^{1-p}_T|\nabla v_0-{\bf v}_g|^p\right)ds,& p\in [1,\infty), \vspace{1mm} \\
      \displaystyle\sup\limits_{T\in{\cal T}_h}\left(  h_T^{-2}\|v_0-v_b\|_{0,\infty,\partial T}+ h_T^{-1} \|\nabla v_0-{\bf v}_g\|_{0,\infty,\partial T}\right), & p=\infty.
   \end{array}\right.
\end{eqnarray*}
 Here $v=\{v_0, v_b, {\bf v}_g\}$ stands for weak finite element functions defined on the
finite element partition ${\cal T}_h$, and $h_e$ denotes the length of  the edge $e\in {\cal E}_h$.
We are now in a position to present a numerical scheme for \eqref{poission} by following the usual primal-dual weak Galerkin finite element method with the new $L^p$ stabilizer.

\begin{algorithm}[$L^p$-Weak Galerkin]
The second-order elliptic equation \eqref{poission} can be
discretized as a constrained optimization problem by seeking $u_h\in S_h$
such that
\begin{equation}\label{min:uh-new}
   u_h={\argmin}_{v\in S_h} s(v),
\end{equation}
where $S_h$ is the set of admissible functions
\begin{equation}\label{sh}
    S_h:=\{v\in V^0_h: \ ({\cal L}_w v, w) = (f, w) \quad \forall w\in W_h\}.
\end{equation}
\end{algorithm}

Note that \eqref{min:uh-new} basically states that $u_h$ is the finite element function which satisfies the given PDE
\eqref{poission} weakly on each element $T$, plus minimizing the mismatch on the boundary of each element in the $L^p$ norm.

For any $\chi_h\in V_h$, denote by $Ds(\chi_h)$ the Fr\'echet
derivative at $\chi_h$. It is not hard to show that the action of the derivative at each $v\in V^0_h$ is given by
\begin{equation}
\begin{split}
\langle Ds(\chi_h),  v\rangle = & \sum_{T\in{\cal T}_h} h_T^{1-2p} \langle
|\chi_0-\chi_b|^{p-1}
sgn(\chi_0-\chi_b),  v_0-v_b\rangle_\pT \\
& + h_T^{1-p} \langle |\nabla \chi_0-\chi_g|^{p-1} sgn(\nabla
\chi_0-\chi_g), \nabla v_0-v_g\rangle_\pT.
\end{split}
\end{equation}
  for $1\le p<\infty$.

By introducing a Lagrange multiplier $\lambda_h\in W_h$, the
constrained minimization problem (\ref{min:uh-new}) can be
formulated as follows: Find $u_h\in V_h^0$ and $\lambda_h\in W_h$ such
that 
\begin{eqnarray}
\langle Ds(u_h), v\rangle + ({\cal L}_w v, \lambda_h) &=& 0, \qquad
\forall v
\in V_h^0\label{EQ:Var:001}\\
({\cal L}_w u_h, \sigma) &=& (f,\sigma), \quad \forall \sigma \in
W_h.\label{EQ:Var:002}
\end{eqnarray}

As for $ p=\infty$, the functional $s(\cdot)$ is not Fr\'echet differentiable so that no traditional Euler-Lagrange formulation is possible for the corresponding weak Galerkin finite element scheme.

\section{Min-Max Characterization}
For simplicity, we introduce the Lagrangian for the constrained optimization problem (\ref{min:uh-new})
$$
J(v,\sigma):= s(v) + b_h(v, \sigma) - (f,\sigma),
$$
where $b_h(v,\sigma):=(\LL_w v, \sigma)$ is a bilinear form.
The equations (\ref{EQ:Var:001})-(\ref{EQ:Var:002}) indicate that the numerical solution $(u_h;\lambda_h)$ is a critical point of the Lagrangian. In this section, we show that this critical point is indeed a saddle point of the same Lagrangian.

Observe that at the critical point $(u_h;\lambda_h)$, the equation \eqref{EQ:Var:002} holds true so that $u_h$ is a weak function in the admissible set $S_h$. Moreover, we have
$$
J(u_h, \lambda_h) = s(u_h).
$$
In fact, for any admissible function $v\in S_h$, one has $J(v,\sigma)= s(v)$ so that
$$
\max_{\sigma\in W_h} J(v,\sigma) = s(v).
$$
It follows that
\begin{equation}\label{May28:001}
J(u_h, \sigma) = s(u_h) = J(u_h, \lambda_h),\qquad \forall \sigma\in W_h.
\end{equation}
As the functional $J(v, \lambda_h)$ is convex in $v$, the condition of $D_v J(u_h, \lambda_h)=0$ implies that $u_h$ is a global minimizer of the functional $v \rightarrow J(v,\lambda_h)$; i.e.,
\begin{equation}\label{May28:002}
J(u_h,\lambda_h)\leq J(v,\lambda_h),\qquad \forall v\in V^0_h.
\end{equation}
Combining \eqref{May28:001} with \eqref{May28:002} yields
\begin{equation}\label{EQ:saddle-point}
J(u_h, \sigma)\leq J(u_h, \lambda_h) \leq  J(v, \lambda_h)\qquad \forall v\in V^0_h, \sigma\in W_h,
\end{equation}
which shows that the critical point $(u_h;\lambda_h)$ is a saddle point of the functional $J(\cdot,\cdot)$.

In summary, the constrained minimization problem (\ref{min:uh-new}) can be formulated as a min-max problem that seeks $u_h$ and $\lambda_h$ such that
$$
(u_h; \lambda_h) = \arg\min_{v\in V^0_h}\arg\max_{\sigma\in W_h} J(v,\sigma).
$$

\section{Solution Existence and Uniqueness}
As a convex minimization problem, the numerical scheme (\ref{min:uh-new}) must have a solution $u_h$ in the corresponding finite element space.  The following is a result on the solution uniqueness.

\begin{theorem}\label{theo:4.1}
For any $p\in (1,\infty)$, the constrained minimization problem (\ref{min:uh-new}) has one and only one solution, provided that the mesh size is sufficiently small.
\end{theorem}

\begin{proof}
Let $u_h^{(1)}$ and $u_h^{(2)}$ be two minimum points. It follows that
\begin{equation*}
   s(u_h^{(j)}) \leq s(v),\qquad \forall v \in S_h.
\end{equation*}
  Note that $s(v)$ is the integral of $|v_0-v_b|^p$. For any two real
numbers $a$ and $b$, one has
$$
|a+b|^p \le 2^{p-1}(|a|^p + |b|^p),
$$
and the equality holds true if and only if $a=b$. The above
inequality is equivalent to
$$
|(a+b)/2|^p \le (|a|^p + |b|^p)/2.
$$
It follows that
\[
s((u_h^{(1)}+u_h^{(2)})/2) \le \left( s(u_h^{(1)})
+s(u_h^{(2)})\right)/2=\min_{v\in S_h} s(v).
\]
Furthermore, the above equality holds  true if and only if
$$
u_0^{(1)}- u_b^{(1)} =u_0^{(2)}- u_b^{(2)},\ \nabla u_0^{(1)}-
u_g^{(1)} = \nabla u_0^{(2)}- u_g^{(2)}.
$$
Or equivalently
\begin{equation}\label{EQ:mynumber:001}
u_0^{(1)}- u_0^{(2)}  =u_b^{(1)} - u_b^{(2)},\ \nabla (u_0^{(1)}
- u_0^{(2)}) = u_g^{(1)} - u_g^{(2)}.
\end{equation}
Next, observe that
$$
({\cal L}_w ( u_h^{(1)} - u_h^{(2)}), \sigma) = 0,\quad \forall \sigma\in W_h
$$
and with the smoothness condition (\ref{EQ:mynumber:001}), the weak action $\partial^2_{ij,w} ( u_h^{(1)} - u_h^{(2)})$ is the same as the strong action $\partial_{ij}^2( u_h^{(1)} - u_h^{(2)})$. It follows that
\begin{equation}\label{EQ:Feb24:001}
({\cal L} ( u_h^{(1)} - u_h^{(2)}), \sigma)= 0.
\end{equation}
Our goal is to show that $e_h=u_h^{(1)} - u_h^{(2)}=0$.
To this end, we first denote by ${\cal Q}_h$ the $L^2$ projection operator from $V^0_h$ onto $W_h$,  and then chose $\sigma={\cal Q}_h({\cal L}e_h)$ in  \eqref{EQ:Feb24:001} to obtain
$$
{\cal Q}_h(\LL e_h) = 0.
$$
It follows that
$$
\LL e_h = (I-{\cal Q}_h) \LL e_h =\sum_{i,j=1}^d (I-{\cal Q}_h) (a_{ij}-\bar{a}_{ij})\partial_{ij}^2 e_h.
$$
  Here and in the following, $I$ denotes the identity operator, and $\bar a_{ij}$ denotes the cell-average of $a_{ij}$, i.e.,
 \begin{equation}\label{aij}
    \bar a_{ij}|_{T}:=\frac{1}{T}\int_{T} a_{ij}.
 \end{equation}
As $a_{ij}$ is continuous on each element $T$, we then have
$$
\|\LL e_h\|_{0,p} \leq C\varepsilon \|\nabla^2 e_h\|_{0,p}.
$$
  Here and in the following, $C$ is a constant independent of the mesh size $h$, which is not necessarily the same at  each appearance.
Then the $W^{2,p}$-regularity assumption implies that
$$
\|e_h\|_{2,p}\leq C \|\LL e_h\|_{0,p}.
$$
Combining the last two estimates gives
$$
\|e_h\|_{2,p}\leq C \varepsilon \|e_h\|_{2,p}.
$$
Thus, we have $e_h\equiv 0$ when the mesh size is sufficiently small.
\end{proof}

\section{Error Estimates} 

We establish in this section some error estimates for the numerical solution.

\subsection{$L^q$-estimate for the Lagrangian multiplier}
This subsection is dedicated to the $L^q$ estimate for the Lagrangian multiplier $\lambda_h$ with $q=p/(p-1)$. To this end, we first introduce some projections operators.

 For each element $T$, we denote by $Q_0$ the $L^2$ projection operator onto $P_k(T), k\ge 2$. For each edge or face $e\subset \partial T$, denote by $Q_b$ and ${\bf Q}_g$ the $L^2$ projection operators onto $P_k(e)$ and $[P_{k-1}(e)]^d$, respectively. For any $v\in H^2(\Omega)$, denote by $Q_hv$ the $L^2$ projection onto the weak finite element space $V_h$ such that on each element $T$,
\[
   Q_hv=\{Q_0v,Q_bv,{\bf Q}_g(\nabla v)\}.
\]
  It has been proved in \cite{WW-2014} that the projection operator $Q_h$ satisfies the following commutative property:
 \begin{equation}\label{QQ}
    \partial^2_{ij,w}(Q_hv)={\cal Q}_h(\partial^2_{ij} v),\ \ i,j=1,\ldots,d.
 \end{equation}
   Here ${\cal Q}_h$ is the $L^2$ projection  operator onto $W_h$.
   
\begin{theorem}
Assume that the coefficient matrix $\{a_{ij}\}_{d\times d}$ in \eqref{poission} is piecewise continuous on each element, and $u_h$ is the numerical solution
arising from (\ref{min:uh-new}) with $p\ge 1$, and $\lambda_h$ is the Lagrangian multiplier for the variational problem. Then the following estimate holds true
 \begin{eqnarray}\label{EQ:Feb23:201}
   \|\lambda_h\|_{0,q}&\le &\left\{
   \begin{array}{ll}
  C s(u_h)^{\frac 1q},&  q\in (1,\infty), \vspace{1mm} \\
  C, & q=1,\infty,
   \end{array}\right.
\end{eqnarray}
provided that the mesh size $h$ is sufficiently small. 
\end{theorem}

\begin{proof}
 Let $P_0\in\Omega$ be the point such that
 \[
    |\lambda_h(P_0)|=\sup_{x\in\Omega} |\lambda_h(x)|,
 \]
 and denote by   $\tilde\delta_{P_0}$ the regularized  delta function associated with $P_0$.  Let
 \begin{eqnarray*}
   \eta&=&\left\{
   \begin{array}{ll}
   |\lambda_h|^{q-1}sgn(\lambda_h),&  q\in [1,\infty), \vspace{1mm} \\
  \tilde\delta_{P_0}sgn(\lambda_h(P_0)), & q=\infty,
   \end{array}\right.
\end{eqnarray*}
and consider the dual problem of seeking $\Phi \in W^{1,p}_0\cap W^{2,p}$ such that
$$
{\cal L} \Phi =\eta.
$$
By using the definition of $Q_h$ and ${\cal Q}_h$ and the commutative property \eqref{QQ}, we have
\begin{eqnarray*}
a_{ij} \partial_{ij,w}^2 Q_h\phi &=& a_{ij} {\cal Q}_h \partial_{ij}^2 \phi
= \bar{a}_{ij} {\cal Q}_h \partial_{ij}^2 \phi + (a_{ij}-\bar{a}_{ij}) {\cal Q}_h \partial_{ij}^2 \phi\\
&=& {\cal Q}_h (a_{ij} \partial_{ij}^2 \phi) + {\cal Q}_h (\bar{a}_{ij} - a_{ij}) \partial_{ij}^2 \phi +  (a_{ij}-\bar{a}_{ij}) {\cal Q}_h \partial_{ij}^2 \phi.
\end{eqnarray*}
 Here $\bar a_{ij}$ denotes the cell average of $a_{ij}$. Thus, for any function $\phi\in W^{2,p}(\Omega)$ and $\sigma\in W_h$,  we  get
\begin{eqnarray*}
(\LL_w Q_h\phi, \sigma) &=& \sum_{i,j=1}^d(a_{ij} \partial_{ij,w}^2 Q_h\phi, \sigma)\\
&=&(\LL\phi, \sigma) + \sum_{i,j=1}^d ((\bar{a}_{ij} -a_{ij})(I-{\cal Q}_h)\partial_{ij}^2 \phi,\sigma).
\end{eqnarray*}
Hence,
\begin{equation}\label{EQ:Feb23:101}
\begin{split}
|(\LL_w Q_h\phi, \sigma)- (\LL\phi, \sigma)|&=|\sum_{i,j=1}^d((\bar{a}_{i,j} - a_{ij})(I-{\cal Q}_h)\partial_{ij}^2 \phi,\sigma)|\\
&\leq C \varepsilon \sum_{i,j=1}^d\|(I-{\cal Q}_h)\partial_{ij}^2 \phi\|_{0,p}\|\sigma\|_{0,q}.
\end{split}
\end{equation}
Here $\varepsilon$ could be sufficiently small as the mesh size decreases.

We now consider the case of $q\in (1,\infty)$. By replacing $\phi$ with $\Phi$, we obtain
 \begin{equation}\label{EQ:Feb23:001}
\begin{split}
\|\lambda_h\|_{0,q}^q &= (\eta, \lambda_h)= ({\cal L} \Phi,\lambda_h) \\
& \leq |({\cal L}_w Q_h \Phi, \lambda_h)| + |({\cal L} \Phi,\lambda_h)-({\cal L}_w Q_h \Phi, \lambda_h)| \\
& \leq |({\cal L}_w Q_h \Phi, \lambda_h)| + C \varepsilon \sum_{i,j=1}^d\|(I-{\cal Q}_h)\partial_{ij}^2 \Phi\|_{0,p} \|\lambda_h\|_{0,q}\\
&\leq |({\cal L}_w Q_h \Phi, \lambda_h)| + C\varepsilon \|\eta\|_{0,p} \|\lambda_h\|_{0,q}\\
&\leq |({\cal L}_w Q_h \Phi, \lambda_h)|+ C\varepsilon  \|\lambda_h\|_{0,q}^{q},
\end{split}
\end{equation}
which leads to
\begin{equation}\label{cc}
(1-C\varepsilon) \|\lambda_h\|_{0,q}^{q} \leq |({\cal L}_w Q_h \Phi, \lambda_h)|=| \langle Ds(u_h), Q_h
\Phi\rangle|.
\end{equation}
 Here in the second step, we have used the  equation (\ref{EQ:Var:001}).
Recall that
\begin{equation*}
\begin{split}
\langle Ds(u_h),  Q_h\Phi\rangle = & \sum_{T\in {\cal T}_h} h_T^{1-2p} \langle
|u_0-u_b|^{p-1}
sgn(u_0-u_b),  Q_0 \Phi- Q_b \Phi \rangle_\pT \\
& + h_T^{1-p} \langle |\nabla u_0- u_g|^{p-1} sgn(\nabla u_0-u_g),
\nabla Q_0 \Phi- Q_b \nabla \Phi\rangle_\pT.
\end{split}
\end{equation*}
For simplicity, we shall deal with the estimate for the first term
on the right-hand side. This can be done by using the usual
H\"older's inequality as follows.
\begin{equation*}
|\langle |u_0-u_b|^{p-1} sgn(u_0-u_b),  Q_0 \Phi- Q_b \Phi
\rangle_\pT| \leq \| u_0-u_b \|_{0,p,\pT}^{p/q} \|Q_0 \Phi- Q_b
\Phi\|_{0,p,\pT}.
\end{equation*}
Summing over all element $T\in {\cal T}_h$ yields
\begin{equation*}
\begin{split}
&\sum_T |\langle |u_0-u_b|^{p-1} sgn(u_0-u_b),  Q_0 \Phi- Q_b \Phi
\rangle_\pT| \\
\leq & \left(\sum_T \| u_0-u_b \|_{0,p,\pT}^p\right)^{1/q} \left(
\sum_T \|Q_0 \Phi- Q_b
\Phi\|_{0,p,\pT}^p\right)^{1/p}\\
\leq & C \left(\sum_T \| u_0-u_b \|_{0,p,\pT}^p\right)^{1/q}
h^{2-\frac1p}
\|\Phi\|_{2,p}\\
\leq & C h^{2-\frac1p} \left(\sum_T \| u_0-u_b
\|_{0,p,\pT}^p\right)^{1/q} \|\lambda_h\|_{0,q}^{q/p}.
\end{split}
\end{equation*}
  Consequently,
 \[
    | \langle Ds(u_h),  Q_h\Phi\rangle|\le C h^{1-2p} h^{2-\frac1p} \left(\sum_T \| u_0-u_b
\|_{0,p,\pT}^p\right)^{1/q} \|\lambda_h\|_{0,q}^{q/p}.
 \]
 Substituting the above error bound into \eqref{cc}, we have
\begin{equation*}
\begin{split}
\|\lambda_h\|_{0,q} & \leq C h^{1-2p} h^{2-\frac1p}\left(\sum_T \|
u_0-u_b
\|_{0,p,\pT}^p\right)^{1/q}\\
& \leq C  \left(\sum_T h^{1-2p} \| u_0-u_b
\|_{0,p,\pT}^p\right)^{1/q}
= C s(u_h)^{1/q}.
\end{split}
\end{equation*}

We next consider the case for $q=1, \infty$. By following the same argument as adopted for \eqref{EQ:Feb23:001}, we obtain
\begin{equation}\label{e:11}
\|\lambda_h\|_{0,q} 
\leq |({\cal L}_w Q_h \Phi, \lambda_h)|+ C\varepsilon  \|\lambda_h\|_{0,q}
=  | b_h( Q_h \Phi, \lambda_h)|+ C\varepsilon  \|\lambda_h\|_{0,q},\ \ q=1,\infty.
\end{equation}
 From the second half of the inequality \eqref{EQ:saddle-point}, we have for all  $c_0\in\mathbb \bR, v\in V_h^0$,
\begin{eqnarray*}
     s(u_h)+b_h(u_h,\lambda_h)\le s(u_h+c_0v)+b_h(u_h+c_0v,\lambda_h),
\end{eqnarray*}
  which yields
\[
    s(u_h)\le s(u_h+c_0v)+b_h(c_0v,\lambda_h),\ \ \forall c_0\in \bR.
\]
Note that the above inequality is valid for all real number $c_0$.
By choosing $c_0=-sgn (b_h(v,\lambda_h))$, we obtain
\begin{eqnarray}\label{ineq:1}
\begin{split}
       |b_h(v,\lambda_h)| &\le  s(u_h- sgn (b_h(v,\lambda_h)) v)-s(u_h) \\
     &\le
      \max( |s(u_h-v)-s(u_h)|, |s(u_h+v)-s(u_h)|)\\
      &\le s(v).
  \end{split}
\end{eqnarray}
Here in the last step, we have used the following triangle inequality 
\[
    |s(u)-s(v)|\le |s(u-v)|,
\]
which holds true for the case of $p=1, \infty$. By the definition of $s(\cdot)$ and the trace inequality, we have
 \[
    s(Q_h\Phi)\le C \|\Phi\|_{2,q}\le C \|\eta\|_{0,q}\le C, \ \ q=1,\infty,
 \]
    which yields, together with \eqref{ineq:1},
 \begin{eqnarray*}
     | b_h( Q_h \Phi, \lambda_h)|\le s(Q_h\Phi) \le  C.
 \end{eqnarray*}
  Plugging the above inequality into \eqref{e:11}, we obtain \eqref{EQ:Feb23:201} for $q=1, \infty$. This completes the proof of the theorem.
\end{proof}

\subsection{Error estimates in a discrete $W^{2,p}$-norm}
For any weak function $v=\{v_0,v_b,{\bf v}_g\}$, define
\begin{eqnarray}\label{tildes}
  \tilde s(v)&= &\left\{
   \begin{array}{ll}
     s(v)^{\frac 1p},&  p\in [1,\infty), \vspace{1mm} \\
    s(v),  & p= \infty.
   \end{array}\right.
\end{eqnarray}
 The following is a discrete $W^{2,p}$-norm for $v$:
 \begin{equation}\label{discrete-H2}
   \|v\|_{2,p,h}=\tilde s(v)+\|{\cal Q}_h {\cal L} v_0\|_{0,p},\ \  p\ge 1.
 \end{equation}

The following result provides a qualitative measure for the ``discontinuity" of the finite element approximation $u_h$ through an  estimation of $s(u_h)$ or $\tilde s(u_h)$.

\begin{theorem}\label{theo:001}
Assume that the exact solution $u$ satisfies $u\in W^{k+1,p}(\Omega)$. For the numerical solution $u_h$
arising from (\ref{min:uh-new}) with $p\ge 1$, the following estimate holds true
\begin{eqnarray}\label{EQ:Feb25:100}
  \tilde s(u_h)\le C h^{k-1} \|u\|_{k+1, p}.
\end{eqnarray}
Furthermore, combining the above estimate with (\ref{EQ:Feb23:201}) yields
\begin{eqnarray*}
\|\lambda_h\|_{0,q}^q &\leq & C h^{(k-1)p} \|u\|_{k+1,p}^{p},\quad q\in (1,\infty),\\
\|\lambda_h\|_{0,q} &\leq & C,\qquad q=1,\infty.
\end{eqnarray*}
\end{theorem}

\begin{proof}
From the saddle-point property of the Lagrangian $J(v,\sigma)$ we have
$$
s(u_h) = J(u_h, \lambda_h) \leq J(v, \lambda_h)=s(v) +(\LL_w v, \lambda_h)-(f,\lambda_h)
$$
for all $v\in V_h^0$. In particular, by choosing $v=Q_h u$ we obtain
\begin{eqnarray*}
s(u_h) &\leq & s(Q_h u) +(\LL_w Q_h u, \lambda_h)-(f,\lambda_h)\\
&=&s(Q_hu) +(\LL_w Q_h u, \lambda_h)-(\LL u,\lambda_h).
\end{eqnarray*}
From the estimate (\ref{EQ:Feb23:101}) we have
$$
s(u_h) \leq s(Q_h u) + C\varepsilon \|(I-{\cal Q}_h)\nabla^2 u\|_{0,p} \|\lambda_h\|_{0,q}.
$$
Substituting the estimate (\ref{EQ:Feb23:201}) to the right-hand side of the above inequality yields
$$
s(u_h) \leq s(Q_h u) + C \varepsilon \|(I-{\cal Q}_h)\nabla^2 u\|_{0,p} s(u_h)^{1/q},\ \ q\in (1,\infty),
$$
  and
$$
s(u_h) \leq s(Q_h u) + C \varepsilon \|(I-{\cal Q}_h)\nabla^2 u\|_{0,p},\ \ q=1,\infty.
$$
Now from the inequality $ab\leq \rho a^p + \rho^{-q/p}b^q$  and the estimate  
$$
\|(I-{\cal Q}_h)\nabla^2 u\|_{0,p}\le C h^{k-1}\|u\|_{k+1,p}
$$ 
we have
 \begin{eqnarray*}
  s(u_h)&\le &\left\{
   \begin{array}{ll}
     s(Q_h u) +  C h^{(k-1)p} \|u\|^p_{k+1, p},&  p\in (1,\infty), \vspace{1mm} \\
    s(Q_h u) + C h^{k-1} \|u\|_{k+1, p},& p= 1, \infty,
   \end{array}\right.
\end{eqnarray*}
  or equivalently,
 \begin{eqnarray}\label{EQ:Feb23:202}
   \tilde s(u_h)\le \tilde s(Q_h u) + C h^{k-1} \|u\|_{k+1, p},\ \  p\ge 1.
 \end{eqnarray}

The term $\tilde s(Q_h u)$ can be handled as follows.
For any $p\in [1,\infty)$, using the trace inequality, we arrive at
\begin{equation*}
\begin{split}
\int_{\pT} |Q_0 u - Q_b u|^p ds & \leq C \int_{\pT} |Q_0 u - u|^p ds
\\
& \leq C h^{-1} \left( \|u - Q_0 u\|_{0, p, T}^p + h^p \|\nabla(u-Q_0 u)\|_{0, p, T}^p\right)\\
&\leq C h^{-1} h^{p(k+1)} \|u\|_{k+1, p, T}^p.
\end{split}
\end{equation*}
Following the same argument, we obtain
\[
\int_{\pT} |\nabla (Q_0 u)-{\bf Q}_b(\nabla u)|^p ds \le C h^{-1} h^{pk} \|u\|_{k+1, p, T}^p.
\]
Thus,
$$
 \tilde s(Q_h u) \leq C h^{(k-1)} \|u\|_{k+1, p},\ \ p\in [1,\infty).
$$
The same analysis can be applied to  the  estimate for the case of $p=\infty$. Substituting the above inequality into (\ref{EQ:Feb23:202}) yields the desired result (\ref{EQ:Feb25:100}). This completes the proof of the theorem.
\end{proof}

\bigskip

Next, we would like to estimate $\|u_h-Q_hu\|_{2,p,h}$, where $Q_h$ is the $L^2$ projection operator into the finite element space $V_h$.  

\begin{theorem}
Assume that $u\in W^{k+1,p}(\Omega)$ and let $u_h$ be its numerical approximation
arising from (\ref{min:uh-new}) with $p\ge 1$. Then, the following estimate holds true
\begin{equation}\label{EQ:Feb25:104}
\|u_h-Q_hu\|_{2,p,h}\leq C h^{k-1}\|u\|_{k+1,p}.
\end{equation}
\end{theorem}

\begin{proof}
First, we note that
 \begin{equation}\label{esti:s}
   \tilde s(u_h-Q_hu)\le \tilde s(u_h)+\tilde s(Q_hu)\le C  h^{k-1}\|u\|_{k+1,p},\ \ \forall p\ge 1.
 \end{equation}
   Second, for any function  $\phi\in L^q(\Omega)$ and weak function $v=\{v_0, v_b, {\bf v}_g\}\in V_h^0$, we have from \eqref{weak:deri} and the integration by parts that
\begin{eqnarray*}
(\LL_w v, \phi)_T&=&\sum_{i,j=1}^d(a_{ij} \partial_{ij,w}^2 v, \phi)_T
=\sum_{i,j=1}^d (\partial_{ij,w}^2 v, {\cal Q}_h( a_{ij}\phi))_T\\
&=&\sum_{i,j=1}^d \left( (\partial_{ij}^2 v_0, {\cal Q}_h( a_{ij}\phi))_T + \langle v_{gi}-\partial_i v_0, {\cal Q}_h( a_{ij}\phi)n_j\rangle_\pT\right)\\
& & \; -\sum_{i,j=1}^d
\langle v_b- v_0, \partial_j ({\cal Q}_h( a_{ij}\phi)) n_i\rangle_\pT\\
&=& (\LL v_0, \phi)_T + \sum_{i,j=1}^d\left(\langle v_{gi}-\partial_i v_0, {\cal Q}_h( a_{ij}\phi)n_j\rangle_\pT
-\langle v_b- v_0, \partial_j ({\cal Q}_h( a_{ij}\phi)) n_i\rangle_\pT\right),
 \end{eqnarray*}
 which gives, together with the trace and inverse inequality,
\begin{equation}\label{EQ:Feb25:106}
| (\LL_w v, \phi) - (\LL v_0, \phi) | \leq C \tilde s(v) \sum_{i,j=1}^d \|{\cal Q}_h(a_{ij}\phi)\|_{0,q}, \ \ \forall \phi\in L^q(\Omega).
\end{equation}
In particular, we have
\begin{eqnarray*}
| ({\cal Q}_h\LL_w v, \phi) - ({\cal Q}_h\LL v_0, \phi) | &=& | (\LL_w v, {\cal Q}_h\phi) - (\LL v_0, {\cal Q}_h\phi) |\\
& \leq & C \tilde s(v) \|\phi\|_{0,q}.
\end{eqnarray*}
By choosing $v=u_h-Q_hu$ in the above inequality we obtain
\begin{equation}\label{uu0}
   |({\cal Q}_h\LL (u_0-Q_0u, \phi) | \le C  | ({\cal Q}_h\LL_w (u_h-Q_hu), \phi) |+ C\tilde s(u_h-Q_hu) \|\phi\|_{0,q}.
\end{equation}
On the other hand, by \eqref{EQ:Var:002} and \eqref{QQ}, we get
\begin{eqnarray}\label{11}
\begin{split}
\left|({\cal Q}_h \LL_w (u_h-Q_h u), \phi)\right|&=\left|(\LL_w (u_h-Q_h u), {\cal Q}_h \phi)\right|&\\
&=\left|(f, {\cal Q}_h\phi) -\sum_{i,j=1}^d (a_{ij} {\cal Q}_h\partial_{ij}^2 u, {\cal Q}_h\phi)\right|&\\
&=\left|\sum_{i,j=1}^d(a_{ij}(I-{\cal Q}_h)\partial_{ij}^2 u, {\cal Q}_h\phi)\right|
\leq C h^{k-1}\|u\|_{k+1,p}\|\phi\|_{0,q}.&
\end{split}
\end{eqnarray}
 Plugging the above estimate and \eqref{esti:s} into \eqref{uu0}, we have
\[
   |({\cal Q}_h\LL (u_0-Q_0u), \phi) | \le C  h^{k-1}\|u\|_{k+1,p}\|\phi\|_{0,q},
\]
which leads to
\[
   \|{\cal Q}_h\LL (u_0-Q_0u)\|_{0,p}\le C h^{k-1}\|u\|_{k+1,p}.
\]
Then the desired result follows. This completes the proof of the theorem.
\end{proof}

\subsection{Error estimates in $L^p$ and $W^{1,p}$}

In this section,  we shall derive error estimates for $u_h$ in the usual $L^p$  and $W^{1,p}$ norms. To this end, consider the auxiliary problem that seeks an unknown function $\varphi$
such that
\begin{eqnarray}\label{dual-problem:new}
\begin{aligned}
   {\LL^*}\varphi:=\sum_{i,j=1}^d\partial_{ji}^2 (a_{ij} \varphi)=\theta,  &&  \mbox{ in} \ \Omega, \\
   \varphi=0,\ \  && \mbox{ on} \ \partial\Omega,
\end{aligned}
\end{eqnarray}
where $\theta$ is a given function in $L^q(\Omega)$. Assume the dual problem (\ref{dual-problem:new}) has the usual $W^{2,q}$-regularity in the sense that the solution is in $W^{2,q}(\Omega)$ and satisfies
\begin{equation}\label{regularity:1}
\|\varphi\|_{2,q}\leq C \|\theta\|_{0,q}.
\end{equation}

 \begin{lemma}\label{lemma:1}
    Assume that the coefficient $a_{ij}\in C^1(\Omega)$, then for any weak function $v=\{v_0,v_b, {\bf v}_g\}$,  and $\varphi\in W^{2,q}(\Omega)$ with $\varphi=0$ on $\partial \Omega$, there holds
  \begin{eqnarray*}
      (v_0, {\LL^*}\varphi)&=&\sum_{i,j=1}^d(\partial_{ij,w}^2 v, a_{ij}\varphi)-\sum_{T\in{\cal T}_h} \sum_{i,j=1}^d\langle
 a_{ij}\varphi-{\cal Q}_h(a_{ij}\varphi),  (\partial_i v_0
 -v_{gi}) n_j \rangle_{\partial T}\\
  & &  + \sum_{T\in{\cal T}_h}\sum_{i,j=1}^d\langle (v_0-v_b)n_i,  \partial_j (a_{ij}\varphi)-\partial_j {\cal Q}_h(a_{ij}\varphi) \rangle_{\partial T}.
\end{eqnarray*}
\end{lemma}
Here we omit the proof since it  has been proved in \cite{WW-mathcomp}.

\begin{theorem}\label{theo:1} Assume $u\in W^{k+1,p}(\Omega)$ and let $u_h$ be its numerical solution arising from (\ref{min:uh-new}) with $p\ge 1$.  Assume that the coefficient $a_{ij}\in C^1(\Omega)$ satisfies $a_{ij}|_{T}\in W^{2,\infty}(T)$ for all $T\in {\cal T}_h, i,j\le d$, and the dual problem \eqref{dual-problem:new} has the $W^{2,q}(\Omega)$ regularity with the a priori estimate \eqref{regularity:1}. Then, the following estimate holds true
\begin{equation}\label{es:2}
   \|u_h-u\|_{0,p}\le C h^{k+1}\|u\|_{k+1,p}, \qquad k\ge 2,
\end{equation}
provided that $P_1(T)\subset W_h(T)$ for all $T\in {\cal T}_h$ and the mesh size $h$ is sufficiently small. In the case that $W_h(T)$ does not contain all the linear functions, the above estimate should  be replaced by
\begin{equation}\label{es:3}
   \|u_h-u\|_{0,p}\le C h^{k}\|u\|_{k+1,p}, \qquad k\ge 2.
\end{equation}
\end{theorem}
\begin{proof}
Set
\[
   e_h=u_h-Q_hu=\{u_0-Q_0u,u_b-Q_bu, {\bf u}_g-{\bf Q}_gu\}=\{e_0,e_b, {\bf e}_g\}.
\]
For any given function $\theta\in L^q(\Omega)$, let $\varphi$ be the solution of \eqref{dual-problem:new}.  From Lemma \ref{lemma:1} we have
 \begin{eqnarray*}
 (\theta, e_0)&=&\sum_{i,j=1}^d(\partial_{ij,w}^2 e_h, a_{ij}\varphi)-\sum_{T\in{\cal T}_h} \sum_{i,j=1}^d\langle
 a_{ij}\varphi-{\cal Q}_h(a_{ij}\varphi),  (\partial_i e_0
 -e_{gi}) n_j \rangle_{\partial T}\\
  & &  + \sum_{T\in{\cal T}_h}\sum_{i,j=1}^d\langle (e_0-e_b)n_i,  \partial_j (a_{ij}\varphi)-\partial_j {\cal Q}_h(a_{ij}\varphi) \rangle_{\partial T}\\
 &=&(\LL_w e_h,  \varphi-{\cal Q}_h\varphi)+(\LL_w e_h, {\cal Q}_h\varphi)+I_3+I_4=I_1+I_2+I_3+I_4.
 \end{eqnarray*}
We next estimate $I_i,i\le 4$  respectively.  By \eqref{EQ:Feb25:106}) and \eqref{EQ:Feb25:104}, and the property of the $L^2$ projection operator ${\cal Q}_h$, we have
\begin{eqnarray*}
|I_1|= |(\LL_w e_h, \varphi-{\cal Q}_h \varphi)| & \lesssim & |(\LL e_0, \varphi-{\cal Q}_h \varphi)|+\tilde s(e_h) \|\varphi-{\cal Q}_h \varphi\|_{0,q}\\
 &\le & C h^{m+1} ( \|\LL e_0\|_{0,p} + \tilde s(e_h)) \|\varphi\|_{m+1,q}\\
&\leq &   C h^{k+m}\|u\|_{k+1,p}\|\varphi\|_{m+1,q}.
 \end{eqnarray*}
Here $m$ is an integer with $0\le m\le 1$.
As to $I_2$, we have from \eqref{QQ}
\begin{eqnarray*}
I_2=(\LL_w e_h, {\cal Q}_h\varphi) &=& (f, {\cal Q}_h \varphi) - (\LL_w Q_h u, {\cal Q}_h \varphi) \\
 &=& (\LL u, Q_h \varphi) -\sum_{i,j=1}^d (a_{ij} {\cal Q}_h\partial_{ij}^2u, {\cal Q}_h \varphi)\\
&=& \sum_{i,j=1}^d(a_{ij} (I - {\cal Q}_h) \partial_{ij}^2 u, {\cal Q}_h \varphi)\\
&=& ((I - {\cal Q}_h) \partial_{ij}^2 u, (I-{\cal Q}_h) (a_{ij} {\cal Q}_h \varphi)).
\end{eqnarray*}
Thus,
\begin{equation}\label{EQ:Feb25:107}
  |I_2|\leq Ch^{m+1} \|(I - {\cal Q}_h) \partial_{ij}^2 u\|_{0,p} \|\varphi\|_{m+1,q}\leq Ch^{k+m}\|u\|_{k+1,p} \|\varphi\|_{m+1,q}.
\end{equation}
   By the Cauchy-Schwartz inequality and the trace inequality,  we have the following estimate for $I_3$ when $p\in [1,\infty)$,
\begin{eqnarray*}
\left|I_3\right| & \lesssim & \left(\sum_{T\in{\cal T}_h} h_T^{(1-p)}\int_{\pT}|\nabla e_0-{\bf e}_g|^pds\right)^{\frac{1}{p}}
  \left(\sum_{T\in{\cal T}_h} h_T^{(p-1)q/p}\int_{\pT}|\varphi-{\cal Q}_h \varphi|^qds\right)^{\frac{1}{q}}\\
&\lesssim &  \|e_h\|_{2,p,h} (\|\varphi-{\cal Q}_h
\varphi\|_{0,q}+h\|\nabla(\varphi-{\cal Q}_h \varphi)\|_{0,q})\\
& \lesssim & h^{k+m}\|u\|_{k+1,p} \|\varphi\|_{m+1,q},
 \end{eqnarray*}
  and for $p=\infty$,
 \begin{eqnarray*}
     \left|I_3\right| & \lesssim & \sup_{T\in{\cal T}_h} h_{T}^{-1}\|\nabla e_0-{\bf e}_g\|_{0,\infty,\partial T}\  \sum_{T\in{\cal T}_h}\int_{\pT}h_T|\varphi-{\cal Q}_h \varphi|ds \\
&\lesssim &  \|e_h\|_{2,\infty,h} (\|\varphi-{\cal Q}_h
\varphi\|_{0,1}+h\|\nabla(\varphi-{\cal Q}_h \varphi)\|_{0,1})\\
& \lesssim & h^{k+m}\|u\|_{k+1,\infty} \|\varphi\|_{m+1,1}.
 \end{eqnarray*}
 The same argument can be applied to yield the following estimate for $I_4$:
 \[
    |I_4|\lesssim  h^{k+m}\|u\|_{k+1,p}\|\varphi\|_{m+1,q},\ \ \forall p\ge 1.
 \]
Combining the above estimates yields
\begin{equation}\label{eq:12}
|(e_0, \theta)| \lesssim h^{k+m}\|u\|_{k+1,p}\|\varphi\|_{m+1,q},\ \ 0\le m\le 1.
\end{equation}
Note that when $P_1(T)\subset W_h(T)$ and  $a_{ij}|_{T}\in W^{2,\infty}(T)$, the above inequality holds true with $m=1$. The desired error estimate \eqref{es:2} then follows from
  the above inequality and the regularity \eqref{regularity:1}. If $W_h(T)$ does not include all the linear functions, then one would only have
\[
     \|{\cal Q}_h \varphi- \varphi\|_{0,q}\le C h\|\varphi\|_{1,q},
\]
which implies that \eqref{eq:12} is valid only with $m=0$ and so is the error estimate \eqref{es:3}.
This completes the proof.
\end{proof}

\begin{theorem}\label{theo:2} Assume the exact solution satisfies
$u\in W^{k+1,p}(\Omega)$ and let $u_h$ be its numerical solution
arising from (\ref{min:uh-new}) with $p\ge 1$.  Assume that the coefficient $a_{ij}\in C^1(\Omega)$, and the dual problem \eqref{dual-problem:new}
has the $W^{2,q}(\Omega)$ regularity with the a priori estimate \eqref{regularity:1}.
Then the following error estimate holds true
\begin{equation}\label{est:3}
   \|\nabla (u_0-u)\|_{0,p}\le C h^{k}\|u\|_{k+1,p}.
\end{equation}
\end{theorem}
\begin{proof} For any given vector field ${\bf \eta}\in [C^1(\Omega)]^d$ with $\eta=0$ on ${\cal E}_h$, let $  \varphi$ be the solution of the dual problem \eqref{dual-problem:new} with
$\theta=-\nabla\cdot \eta$. It follows that
\begin{eqnarray*}
   (\nabla e_0, {\bf \eta})=-(e_0, \nabla\cdot {\bf \eta})=(e_0, \theta)
      =I_1+I_2+I_3+I_4,
\end{eqnarray*}
Here $I_i, i\le 4$ are exactly the same as in the proof of Theorem \ref{theo:1}. By  choosing $m=0$  in \eqref{eq:12}, we get
  \[
      |(\nabla e_0, {\bf \eta}) |\le C h^{k}\|u\|_{k+1,p}\|\varphi\|_{1,q}\le C  h^{k}\|u\|_{k+1,p} \|\eta\|_{0,q}.
  \]
As the set of all such vector fields $\eta$ is dense in $[L^q(\Omega)]^d$, we thus have
  \[
     \|\nabla e_0\|_{0,p}\le C h^{k}\|u\|_{k+1,p}.
  \]
    The error estimate \eqref{est:3} then follows from the triangle inequality and the approximation property of the $L^2$ projection operator. This completes the proof.
\end{proof}

\section{Algorithm and Implementation}

In this section, we develop a fixed-point iterative algorithm for solving the minimization problem \eqref{min:uh-new} based on the $L^p$ stabilizer for $p\neq 2$. For simplicity,  we present only the case $p=1$. Other cases can be handled in similar manners.

\subsection{An equivalent matrix form of the minimization problem}
We  first  reformulate the constrained minimization problem \eqref{min:uh-new} as an equivalent discrete minimization problem in $\bR^{N}$.

Recalling the definition of $V^0_h$ in \eqref{Vh}-\eqref{vh0}, a function $v_h\in V^0_h$ can be represented as a triple 
\[
    v_h=\{v_0,v_b,{\bf{v_g}}\},\ {\rm with}\ 
    {\bf{v_g}}=(v_{g1},\ldots,v_{gd}).
\]
Suppose that we have chosen bases $\{\phi_{0,n}\}_{n=1}^{N_1}$,$\{\phi_{b,n}\}_{n=1}^{N_2}$, $\{\phi^j_{g,n}\}_{n=1}^{N_3}$ for $v_0,v_b, v_{gj}, 1\le j\le d$, respectively. With the bases, we have that
\begin{equation}\label{6.1}
v_0=\sum_{i=1}^{N_1}v_{0,i}\phi_{0,i},\ \ 
v_b=\sum_{i=1}^{N_2}v_{b,i}\phi_{b,i},\ \ 
 {v_{gj}}=\sum_{i=1}^{N_3} v^j_{g,i} \phi^j_{g,i},
\end{equation}
where $v_{0,i},v_{b,i},v^j_{g,i}$ are real numbers.
Define $N:=N_1+N_1+dN_3$ and 
\begin{equation}\label{6.2}
    {\bf v}:=(v_1,\ldots,v_N)^T=( {\bf v}_0,{\bf v}_b,{\bf v}_{g1},\ldots, {\bf v}_{gd})^T,
\end{equation}
   where 
\begin{equation}\label{6.3}
    {\bf v}_0=(v_{0,1},\ldots, v_{0,N_1}),\ 
    {\bf v}_b=(v_{b,1},\ldots, v_{b,N_2}),\ 
    {\bf v}_{gj}=(v^j_{g,1},\ldots, v^j_{g,N_3}),\ 1\le j\le d. 
\end{equation}
   We call ${\bf v}$  the coefficient vector of $v_h.$
  Clearly, the vector ${\bf v}$ is uniquely 
  determined by the function $v_h$ and 
  vice versa. 
 Let 
\[
{\bf\phi}:=(\phi_1,\ldots,\phi_n)^T=(\{\phi_{0,n}\}_{n=1}^{N_1},
    \{\phi_{b,n}\}_{n=1}^{N_2}, \{\phi^j_{g,n}\}_{n=1}^{N_3})^T,
\]
 and $\{\psi_m\}_{m=1}^M$ be a basis of  $W_h$. With the definition of $S_h$ in \eqref{sh}, a function $v_h\in S_h\subset V_h^0$ has the following variational equation
\[
   ({\cal L}_w v_h, \psi_m) = (f, \psi_m),\ \  \mbox{for all} \ \  m=1,2,\ldots M.
\]
We introduce the matrix
\begin{equation}\label{a}
 A:=(b_{m,n})_{M\times N},\ \ \mbox{with}\ \ b_{m,n}:=\sum_{i,j=1}^d(a_{ij}{\cal L}_w\phi_n,\psi_m),\ \ 1\le m\le M, 1\le n\le N.
\end{equation}
The variational equation above becomes
\begin{equation}\label{af}
    A {\bf{v}}={\bf{f}}, \ \ \mbox{where}\ \ {\bf{f}}:=(f_1,\ldots,f_M)^{T},\ \ \mbox{with} \ \ f_m:=(f,\psi_m),\ \ 1\le m\le M,
\end{equation}
In other words, the function space  $S_h$ defined in \eqref{sh} can be represented as
\[
   S_h=\left\{v_h\in V_h^0:  A\bf v=f\right\}.
\]
Note that the discrete bilinear form $({\cal L}_w v_h, \sigma)$ satisfies the {\it inf-sup} condition (see \cite{WW-mathcomp})
and thus, the matrix $A$ is of row full rank. 

We denote by $N_e$ the cardinality of ${\cal E}_h$ (i.e., the number of all edges).
For any function $v_h:=\{v_0,v_b,{\bf v}_g\}$, we note that 
on each edge $e_i\in \partial T, T\in{\cal T}_h$, $(v_0-v_b)(x)$  and $(\partial_j v_0-{\bf v}_{gj})(x)$, $1\le j\le d$,  are polynomials of degree $k$, that is, $(v_0-v_b)(x), (\partial_j v_0-{\bf v}_{gj})(x) \in {\mathbb P}_{k}(e_i)$. By a scaling from $e_i$ to the reference element $(0,1)$, 
it follows that
\[
h_{T}^{1-2p}(v_0-v_b)(x)|_{e_i} =(h_{e_i})^{-1}\sum_{m=0}^{k}c_{i,m+1}t^{m}, \ \ \mbox{for all}\ \  T\in {\cal T}_h, e_i\in\partial T, t\in (0,1),
\]
where $h_{e_i}$ denotes the length of $e_i$ and
the coefficients $c_{i,m}$ are constants dependent on ${\bf v}:=(v_1,\ldots,v_N)^{T}$.
Denoting by $\bar {\bf v}_i$ the vector representation of the function $v_0-v_b$ on 
$e_i$ and  
defining  $c_i:=(c_{i,1},\ldots,c_{i,k+1})$, 
 we introduce a linear local operator ${\cal \bar B}_i: {\mathbb R}^{N_1+N_2}\rightarrow  {\mathbb R}^{k+1}$ by
\[
{\cal \bar B}_i {\bf \bar v}_i={\bf c}_i,\ 1\le i\le N_e.
\]
Let $\bar B_i\in \bR^{(k+1)\times {(N_1+N_2)}}$ be the matrix representation of the operator ${\cal \bar B}_i$. Then, we have that
\[
   \bar B_i {\bf \bar v}_i={\bf c}_i. 
\]
where $\bar B_i=h_{e_i}^{-1}h_{T}^{2p-1}I$  with $I$ the identity matrix for 
  $e_i\in\partial\Omega$, and 
 \begin{equation}\label{Bi}
 \bar B_i=h_{e_i}^{-1}h_{T}^{2p-1}\left( \begin{array}{cccccccc}
  1& 0& \cdots& 0& -1 &0&\cdots& 0\\
  0 & 1& \cdots& 0& 0 &-1&\cdots& 0\\
 \vdots&\vdots&\ddots& \vdots&\ddots&\vdots \vdots&\vdots&\\
  0&0& \cdots& 1& 0& 0&\cdots&-1
\end{array}
\right )
\end{equation}
 for $e_i\in {\cal E}_h^0$. 
Defining ${\bf c}:=(c_1,\ldots,c_{N_e})^{T}$ and going through all element $T\in{\cal T}_h$ and all edges of $T$, we obtain a global matrix $B_0\in\bR^{(k+1)N_e\times N}$ such that 
\[
B_0 {\bf v} ={\bf c}
\]
 with $B_0$ dependent on the local matrix $\bar B_i$. Likewise,  for each $j$ with $1\le j\le d$,  it follows that
\[
   h_{T}^{1-p}(\partial_j v_0-{\bf v}_{gj})(x)|_{e_i} =(h_{e_i})^{-1}\sum_{m=0}^{k}d^{j}_{i,m+1}t^{m},\ \ \mbox{for all}\ \ T\in {\cal T}_h, e_i\in\partial T, t\in (0,1). 
\]
Then we can construct the matrix $B_j$ that satisfies
\[
    B_j{\bf v} ={\bf d}^j,
\]
where ${\bf d}^j:=(d^j_1,\ldots,d^j_{N_e})^{T}$ and  $d^j_i:=(d^{j}_{i,1},\ldots,d^{j}_{i,k+1})$.

We are now ready to rewrite the constrained minimization problem \eqref{min:uh-new} in its equivalent discrete form. To this end, we define a convex function $\varphi: \bR^{(d+1)(k+1)N_e}\rightarrow \bR$ by
\begin{equation}\label{varphi}
\varphi({\bf q}):=\sum_{i=1}^{(d+1)N_e}\int_{0}^1\left|\sum_{m=0}^{k}q_{i,m+1}t^{m}\right|dt,
\end{equation}
for ${\bf q}:=(q_1,\ldots,q_{(d+1)N_e})^{T}$, with $q_i:=(q_{i,1},\ldots,q_{i,k+1})$.
By defining
\[
   B:=(B^T_0,B^T_1,\ldots,B^T_d)^T,
\]
from the definition of $s(\cdot)$ for $p=1$,  we identify $s(v_h)$ as a composition of $\varphi$ with $B$, that is,
\begin{equation}\label{ss1}
s(v_h)=\varphi (B{\bf v}).
\end{equation}
Thus, it can be verified that the constrained minimization problem \eqref{min:uh-new} is equivalent to
\begin{equation}\label{discrete-form}
   {\bf u}={\arg\min}_{\bf v\in \bR^{N}, A\bf v=\bf f} (\varphi\circ B)({\bf v}).
\end{equation}
Let $L$ denote the indicator function, defined at ${\bf x}\in {\mathbb R}^M$ by
\begin{eqnarray}\label{indicator-function}
   L({\bf x})&=&\left\{
   \begin{array}{ll}
   0, & \ {\rm if}\ {\bf x}={\bf f}, \vspace{1mm}  \\
   +\infty, & \ {\rm otherwise}.
   \end{array}\right.
\end{eqnarray}
The minimization problem \eqref{discrete-form} may be rewritten as
 \begin{equation}\label{2.5}
   {\bf u}=\arg\min_{\bf v\in \bR^{N}}\left\{ \varphi( B{\bf v})+L(A{\bf v})\right\}.
\end{equation}
Note that the domain of the indicator function $L$ is convex, which indicates that $L$ is convex.
Consequently,  \eqref{2.5} can be described as the
minimum point of the sum of two convex functions. Actually, optimization problems which minimizes the sum of
two convex functions has important applications in image and signal processing, and has been widely discussed and studied during
the past decades (see, e.g., \cite{a1,krol-li-shen-xu, li-shen-xu-zhang, li-Micchelli-sheng-xu, Micchelli-Shen-Xu, Rudin-Osher}).

\subsection{Characterization via a system of fixed-point equations}
In this subsection, we characterize the solution of \eqref{2.5} in term
of a system of fixed-point equations via the proximity operator of $\varphi$ and $L$. The system of fixed-point equations will serves as a basis for developing our iterative algorithm.

Noticing that the definition of function $\varphi$ involves the absolute value function, it is not differentiable and neither is the indicator function $L$. Thus, the standard gradient type methods are not applicable to the minimization problem \eqref{2.5}. Inspired by the work \cite{krol-li-shen-xu, li-shen-xu-zhang, li-Micchelli-sheng-xu, Micchelli-Shen-Xu}, we shall develop a fixed-point iterative algorithm for solving \eqref{2.5}, without using gradient information of the functions $\varphi$ and $L$. 

We begin with recalling the notion of the subdifferential and the proximity operator of a convex function.

\begin{Defi}
Let $\psi$ be a real-valued convex function on $\mathbb{R}^m$.
The subdifferential of $\psi$ at ${\bf x}\in \bR^m$ is defined by
\[
\partial\psi({\bf x}):=\{{\bf y}: {\bf y}\in \bR^m, \ {\rm and }\ \psi({\bf z})\ge \psi({\bf x})+({\bf y},{\bf z}-{\bf x}),\ \mbox{for all}\ {\bf z}\in \bR^m\},
\]
and the proximity operator of $\psi$ is defined for ${\bf x}\in \bR^m$ by
\[
{\rm prox}_{\psi}({\bf x}):={\arg\min}\{\frac{1}{2}({\bf v}-{\bf x}, {\bf v}-{\bf x})+\psi({\bf v}): {\bf v}\in \bR^m\}.
\]
Here $(\cdot,\cdot)$ denotes the standard inner product on $\mathbb{R}^m$.
\end{Defi}

The subdifferential of a convex function is intimately related to its proximity operator. According to \cite{Bauschke-Combettes, Micchelli-Shen-Xu}, we have the following relationship between the subdifferential and the proximity operator of a convex function:
\begin{equation}\label{relation}
{\bf y}\in \partial\psi({\bf x})\ \ \mbox{if and only if}\ \ {\bf x}={\rm prox}_{\psi}({\bf x}+{\bf y}).
\end{equation}


Appealing to relation \eqref{relation}, we  characterize the solution of \eqref{2.5} in terms of a system of fixed-point equations via the proximity operator of $\varphi$ and $L$.

\begin{proposition} \label{lemma:10}
If ${\bf u}\in {\mathbb R}^N$ is the solution of problem \eqref{2.5} with  $A\in {\mathbb R}^{M\times N}$ and $B\in {\mathbb R}^{(d+1)(k+1)N_e\times N}$,
then there exists  vectors ${\bf x}\in  {\mathbb R}^M$ and ${\bf y}\in {\mathbb R}^{(d+1)(k+1)N_e}$ such that for any positive constants $\alpha, \beta$,
\begin{eqnarray}\label{2.90}
\left\{
\begin{array}{ll}
    \beta A^T{\bf x}+ \alpha B^T{\bf y}=0, &\\
    {\bf y}=({\cal I}-{\rm prox}_{\frac{1}{\alpha}\varphi})(B{\bf u}+{\bf y}), &\\
     {\bf x}=({\cal I}-{\rm prox}_{\frac{1}{\beta} L} )(A{\bf u}+{\bf x}), &
\end{array}\right.
\end{eqnarray}
where ${\cal I}$ denotes the identity operator.  Conversely, if there exist positive constants $\alpha, \beta$
     and  vectors ${\bf u}\in {\mathbb R}^N$, ${\bf x}\in {\mathbb R}^M$, ${\bf y}\in {\mathbb R}^{(d+1)(k+1)N_e}$ satisfying
    \eqref{2.90}, then ${\bf u}$ is the solution of \eqref{2.5}.
\end{proposition}
\begin{proof}
Let ${\bf u}\in {\mathbb R}^N$ be the solution of problem \eqref{2.5}. According to the Fermat rule and the definition of the subdifferential, we have that
 \[
 {\bf 0}\in\partial(\varphi\circ B)({\bf u})+ \partial(L\circ A)({\bf u}).
 \]
It follows that there exist  ${\bf y_0}\in \partial(\varphi\circ B)({\bf u})$ and ${\bf y_1}\in   \partial (L\circ A)({\bf u})$ such that
${\bf y_0}+ {\bf y_1}=0$.
The chain rule
$$
\partial (\varphi\circ B)({\bf u})=B^T\partial \varphi(B {\bf u} )
$$
ensures that there exists an
\begin{equation}\label{x_0}
{\bf x}_0\in \partial\frac{1}{\alpha}\varphi(B {\bf u} )
\end{equation}
satisfying ${\bf y_0}=\alpha B^T{\bf x}_0$.
Likewise, by the chain rule
$$
\partial (L\circ A)({\bf u})=A^T\partial L(A {\bf u} ),
$$
there exists an
\begin{equation}\label{x_1}
{\bf x}_1\in \partial\frac{1}{\beta}L(A {\bf u} )
\end{equation}
such that ${\bf y_1}=\beta A^T{\bf x}_1$.
Consequently, we obtain that
\begin{equation}\label{a:00}
\alpha  B^T{\bf x}_0 +  \beta A^T{\bf x}_1 ={\bf 0}.
\end{equation}
By employing the equivalence relation \eqref{relation}, we observe that inclusions \eqref{x_0} and \eqref{x_1} are equivalent to
$$
B{\bf u}={\rm prox}_{ \frac{1}{\alpha}\varphi} (B{\bf u}+{\bf x}_0)
$$
and
$$
A{\bf u}={\rm prox}_{ \frac{1}{\beta}L} (A{\bf u}+{\bf x}_1),
$$
respectively. Or equivalently, we have that
\begin{equation}\label{AAU0}
{\bf x}_0=({\cal I}-{\rm prox}_{ \frac{1}{\alpha}\varphi})(B{\bf u}+{\bf x}_0).
\end{equation}
and
\begin{equation}\label{AAU1}
{\bf x}_1=({\cal I}-{\rm prox}_{ \frac{1}{\beta}L})(A{\bf u}+{\bf x}_1).
\end{equation}
Thus, system \eqref{2.90} of the  fixed-point equations follows from \eqref{a:00}, \eqref{AAU0} and \eqref{AAU1}.

Conversely, suppose that there exist  positive constants $\alpha, \beta$ and  vectors ${\bf u}\in {\mathbb R}^N$, ${\bf x}\in {\mathbb R}^M$, ${\bf y}\in {\mathbb R}^{(d+1)(k+1)N_e}$ satisfying the system \eqref{2.90}. According to the equivalence relation \eqref{relation}, we conclude that $\alpha {\bf y}\in  \partial \varphi  (B{\bf u})$ and $\beta {\bf x}\in \partial L (A{\bf u})$. Consequently,
\[
{\bf 0}=\alpha B^T{\bf y}+\beta A^T{\bf x}\in B^T\partial {\varphi} (B{\bf u})+ A^T\partial L (A{\bf u}).
\]
Again, using the chain rule, we obtain that
\[
{\bf 0}\in \partial (\varphi\circ B)({\bf u})+\partial (L\circ A)({\bf u}).
\]
That is,  the zero vector is in the subdifferential of the objective function at ${\bf u}$. By the Fermat rule, ${\bf u}$ is the solution of \eqref{2.5}.
\end{proof}

\subsection{Iterative algorithm}

In light of the conclusion in Proposition \ref{lemma:10},  we design an iterative algorithm for finding the solution of \eqref{2.5} based on the system \eqref{2.90} of fixed-point equations.

Suppose that initial guesses  ${\bf u}^0\in {\mathbb R}^N$, ${\bf x}^0\in {\mathbb R}^M$ and ${\bf y}^0\in {\mathbb R}^{(d+1)(k+1)N_e}$ are chosen. We construct an implicit iterative scheme from \eqref{2.90} as follows:
\begin{eqnarray}\label{al:1}
\left\{
   \begin{array}{ll}
    \beta A^T{\bf x}^{n+1}+\alpha  B^T{\bf y}^{n+1}=0, &\\
    {\bf y}^{n+1}=  B{\bf u}^{n+1}-B {\bf u}^n+({\cal I}-{\rm prox}_{\frac{1}{\alpha}\varphi})(B{\bf u}^{n}+ {\bf y}^n), &\\
     {\bf x}^{n+1}=({\cal I}-{\rm prox}_{\frac{1}{\beta} L} )(A{\bf u}^{n+1}+{\bf x}^{n+1}). &
   \end{array}\right.
\end{eqnarray}
We next re-express \eqref{al:1} in a compact form for simplicity. For a non-negative integer $n$ we define
\begin{equation}\label{vv}
      {\bf v}^{n}: =({\bf y}^n,{\bf u}^n, {\bf x}^n)^T.
\end{equation}
As shown by \eqref{al:1}, we can obtain ${\bf v}^{n+1}$ as long as  ${\bf v}^{n}$ is available by solving a linear system.
To present the linear system,  we first rewrite the third equation of \eqref{al:1} into its equivalent form. 
Note that the third equation of \eqref{al:1} involves the proximity of the  indicator function $L$, which is defined in \eqref{indicator-function}. A direct calculation from \eqref{indicator-function} and the  definition of the proximity operator confirms that
\begin{equation}\label{Au:1}
{\rm prox}_{\frac{1}{\beta} L}({\bf x} )={\bf f},\ \ \mbox{for all}\ \  {\bf x},
\end{equation}
where ${\bf f}$ is given by \eqref{af}.

We next define 
\begin{equation}\label{M1}
S: =\left( \begin{array}{ccc}
  I & -B&   {\bf 0}\\
 {\bf 0} & \alpha B^TB &  \beta A^T \\
 {\bf 0} &\beta A & {\bf 0} \\
 \end{array}
\right ),
\end{equation}
where  $I$ is the identity matrix of order $(d+1)(k+1)N_e$.

\begin{lemma} Let $A$ be a  row full rank matrix defined in \eqref{a}, and $B$ be the  matrix satisfying \eqref{ss1}.  Then for any positive integers
$\alpha$ and $\beta$,  the matrix $S$ defined in \eqref{M1} is non-singular. 
\end{lemma}
 \begin{proof}  Let 
\begin{equation}
D: =\left( \begin{array}{cc}
    \alpha B^TB &  \beta A^T \\
    \beta A & {\bf 0} \\
 \end{array}
\right ).
\end{equation}
It suffices to show that $D$ is invertible. To this end,  we prove that the equation 
\begin{equation}\label{new:3}
    D{\bf c}={\bf 0}
\end{equation}
where 
\[
{\bf c}:=({\bf u},{\bf x})^T,\ \ {\bf u}:=(u_1,\ldots, u_N)^T,\ \ {\bf x}:=(x_1,\ldots, x_M)^T, 
\] 
has only trivial solution.   
It follows from \eqref{new:3} that 
\begin{equation}\label{new:2}
    \alpha B^TB{\bf u}+\beta A^T{\bf x}={\bf 0}, \ \ \beta Au=0. 
\end{equation}
Multiplying the both sides of the first equation of \eqref{new:2} by ${\bf u}^T$ yields 
\[
\alpha {\bf u}^TB^TB{\bf u}+\beta{\bf u}^T A^T{\bf x}=\alpha {\bf u}^TB^TB{\bf u}+\beta(A{\bf u})^T{\bf x}={\bf 0}. 
\] 
Using $A{\bf u}=0$ in the resulting equation, we get that
\[
{\bf u}^TB^TB{\bf u}={\bf 0}. 
\]
That is equivalent to $B{\bf u}={\bf 0}$, which will be used together with $A{\bf u}={\bf 0}$ to conclude that ${\bf u}={\bf 0}$. 


Given a vector ${\bf u}=(u_1,\ldots u_N)^T$, from \eqref{6.1}-\eqref{6.3} there exists a function $\bar u_h\in V_h^0$ with the coefficient vector ${\bf u}$. Since $B{\bf u}$={\bf 0}, we have that
$\varphi(B{\bf u})=0$, where $\varphi$ is defined by \eqref{varphi}. Then it follows from \eqref{ss1} that 
$s(\bar u_h)=0$. Moreover, using the definition \eqref{a} of matrix $A$ and the equation  $A{\bf u}=0$ leads to 
\[
({\cal L}_w \bar u_h, \sigma)=0,\ \ \mbox{for all}\ \ \sigma\in W_h. 
\]
This implies that 
\[
s(\bar u_h)=0\le  s(v_h),\ \ \mbox{for all}\ \ v_h\in S_h.
\]
In other words,  the function $\bar u_h$ is the minimizer of  \eqref{min:uh-new}.  As the solution of the optimization problem \eqref{min:uh-new} is unique and $s(\bar u_h)=0$, we have that
$\bar u_h=0$, and thus ${\bf u}:=(u_1,\ldots, u_N)^T={\bf 0}$. 
   
It remains to show that ${\bf x}={\bf 0}$. To this end, we substitute the  equation $B{\bf u}={\bf 0}$ into the first equation of \eqref{new:2}, and obtain that 
\[
A^T {\bf x}={\bf 0}. 
\]   
Since the matrix $A$ is of row full rank, the above equation has only the trivial solution. Thus, we have that ${\bf x}=0$. 
 
We have shown that equation \eqref{new:3} has only zero solution. Consequently,  $D$ is  invertible. 
\end{proof}

We rewrite the equations of \eqref{al:1} in the next proposition.

\begin{proposition}\label{Linear-system}
Let $S$ be the block matrix defined in \eqref{M1}.
For given vectors ${\bf u}^n\in {\mathbb R}^N$, ${\bf x}^n\in {\mathbb R}^M$, ${\bf y}\in {\mathbb R}^{(d+1)(k+1)N_e}$, ${\bf v}^{n}$ is defined as in \eqref{vv} and let
\begin{equation}\label{M}
{\bf b}^n :=
\left( \begin{array}{c}
      -B {\bf u}^n+({\cal I}-{\rm prox}_{\frac{1}{\alpha}\varphi})(B{\bf u}^{n}+ {\bf y}^n)\\
\alpha B^T{\rm prox}_{\frac{1}{\alpha}\varphi}(B{\bf u}^{n}+ {\bf y}^n)+\beta A^T{\bf x}^n  \\
 \beta  {\bf f} \\
\end{array}
\right ).
\end{equation}
Then \eqref{al:1} has the following equivalent form
\begin{equation}\label{M2}
 S{\bf v}^{n+1}={\bf b}^n.
\end{equation}
\end{proposition}
\begin{proof}
Substituting the second equation of \eqref{al:1} into the first equation, we get that
\[
\beta A^T{\bf x}^{n+1}+\alpha  B^TB{\bf u}^{n+1}= \alpha B^T{\rm prox}_{\frac{1}{\alpha}\varphi}(B{\bf u}^{n}+ {\bf y}^n)+\beta A^T{\bf x}^n.
\]
By \eqref{Au:1}, the third equation of \eqref{al:1} can be rewritten as
\[
A {\bf u}^{n+1}= {\bf f}.
\]
Consequently, \eqref{al:1} is equivalent to
\begin{eqnarray}\nonumber\label{al:2}
\left\{
   \begin{array}{ll}
     \beta A^T{\bf x}^{n+1}+\alpha  B^TB{\bf u}^{n+1}= \alpha B^T{\rm prox}_{\frac{1}{\alpha}\varphi}(B{\bf u}^{n}+ {\bf y}^n)+\beta A^T{\bf x}^n, &\\
    {\bf y}^{n+1}-B{\bf u}^{n+1}=-B{\bf u}^{n}+({\cal I}- {\rm prox}_{\frac{1}{\alpha}\varphi})(B{\bf u}^{n}+{\bf y}^n), &\\
   A {\bf u}^{n+1}= {\bf f}. &
   \end{array}\right.
\end{eqnarray}
Thus, \eqref{M2} follows.
\end{proof}

Proposition \ref{Linear-system} gives rise to our algorithm.

\begin{algorithm}\label{ALGORITH}
Choose ${\bf v}^0:=({\bf y}^0,{\bf u}^0, {\bf x}^0)$. For $n=0, 1,  \dots$, do:
\begin{itemize}
\item[1.] Evaluate the proximity operator  $prox_{\frac{1}{\alpha}\varphi}$ at $B{\bf u}^{n}+{\bf y}^n$ and compute ${\bf b}^n$ by \eqref{M};
\item[2.]  Solve the linear system 
$$
S{\bf v}^{n+1}={\bf b}^n
$$ 
for ${\bf v}^{n+1}$, with $S$ being defined by \eqref{M1}.
\end{itemize}
\end{algorithm}

The implementation of Algorithm \ref{ALGORITH} requires computing 
the proximity operator of $\varphi$, which will be presented in section 8.

\section{Convergence Analysis of the Iterative Algorithm}

This  section is devoted to the convergence analysis of the iterative Algorithm \eqref{al:1}.
  
We will need a property of the proximity operator, which review below.
An operator  ${\cal P}: {\mathbb R}^d\rightarrow {\mathbb R}^d$ is called {\it nonexpansive} if 
\[
\|{\cal P}({\bf x})-{\cal P}({\bf y})\|\le \|{\bf x}-{\bf y}\|,\ \ \mbox{for all}\ \ {\bf x}, {\bf y}\in {\mathbb R}^d.
\]
It is well-known (cf. \cite{Combettes-Wajs1}) that the proximity operator of a convex function $\psi$ satisfies the following inequality for all ${\bf x},{\bf y}\in {\mathbb R}^d$
\begin{equation}\label{prox:ineq}
\|{\rm prox}_{\psi}({\bf x})-{\rm prox}_{\psi}({\bf y})\|^2\le ({\bf x}-{\bf y}, {\rm prox}_{\psi}({\bf x})-{\rm prox}_{\psi}({\bf y})).
\end{equation}
In other words, the proximity operator is firmly nonexpansive and thus, it is nonexpansive.  As a result of \eqref{prox:ineq}, we conclude that
${\cal I}-\alpha {\rm prox}_{\psi}$ is also firmly nonexpansive for all $0\le \alpha \le 1$. Here ${\cal I}$ denotes the identity operator.

We will frequently use two technical identities in our convergence analysis of Algorithm \eqref{al:1}. For any vector ${\bf a}\in {\mathbb R}^d$, we define its $\ell_2$-norm by $\|{\bf a}\|:=({\bf a},{\bf a})$. A direct calculation confirms that
\begin{equation}\label{aa:1}
 2({\bf a}+{\bf b},{\bf a})= \|{\bf a}\|^2+\|{\bf a}+{\bf b}\|^2-\| {\bf b}\|^2,\ \ 2({\bf a}-{\bf b},{\bf a})= \|{\bf a}\|^2+\|{\bf a}- {\bf b}\|^2-\|{\bf b}\|^2.
\end{equation}

\begin{lemma} \label{Lemma7.1}
The iterative scheme \eqref{al:1} is stable in the sense that
\begin{equation}\label{a:2}
\|{\bf y}^{n+1}\|^2+\|B{\bf u}^{n+1}\|^2+\sum_{i=0}^n\big(\| B {\bf u}^{i}-B{\bf u}^{i+1}\|^2+\| {\bf y}^{i+1}-{\bf y}^{i}\|^2\big)\le  \|{\bf y}^{0}\|^2+\|B{\bf u}^{0}\|^2, 
\end{equation}
for all nonnegative integers $n$, where ${\bf u} ^{n}$ and ${\bf y} ^{n}$ are  sequences generated from \eqref{al:1}. Moreover, there holds 
\begin{equation}\label{a00:3}
\lim_{n\rightarrow +\infty} (\| B {\bf u}^{n}-B{\bf u}^{n+1}\|^2+\| {\bf y}^{n+1}-{\bf y}^{n}\|^2)=0. 
\end{equation}
\end{lemma}
\begin{proof} 
Since ${\cal I}- {\rm prox}_{\frac{1}{\alpha} \varphi}$ is  nonexpansive,  we have from \eqref{prox:ineq} and the second equation of \eqref{al:1} that for $i=0,1,\dots$,
\begin{eqnarray*}
({\bf y}^{i+1}+B{\bf u}^{i}-B {\bf u}^{i+1},  {\bf y}^{i+1}+B{\bf u}^{i}-B {\bf u}^{i+1})\le ( {\bf y}^{i+1}+B{\bf u}^{i}-B {\bf u}^{i+1},  {\bf y}^{i}+B{\bf u}^{i}).
\end{eqnarray*}
This is equivalent to 
$$
({\bf y}^{i+1}+B{\bf u}^{i}-B {\bf u}^{i+1}, {\bf y}^{i+1}-B {\bf u}^{i+1}-{\bf y}^{i})\leq 0, \ \ \mbox{for}\ \ i=0,1,\dots, n.
$$
For $i=0,1,\dots,n$, letting 
\[
I_i:=(B{\bf u}^{i}-B {\bf u}^{i+1},{\bf y}^{i+1}-{\bf y}^{i})-({\bf y}^{i+1}, B {\bf u}^{i+1}),
\]
from the inequality above we have that
\begin{eqnarray}\label{eq:18}
({\bf y}^{i+1}, {\bf y}^{i+1}-{\bf y}^{i})+(B{\bf u}^{i+1}-B {\bf u}^{i}, B{\bf u}^{i+1})+I_i\leq 0,  \ \ \mbox{for}\ \ i=0,1,\dots, n.
\end{eqnarray}

We next show that $I_i\geq 0$ for $i=0,1,\dots,n$. From \eqref{Au:1} and the third equation  of \eqref{al:1} we obtain that
\begin{equation}\label{Au:2}
A{\bf u}^{i+1}={\bf f}, \ \ \mbox{for}\ \ i=0,1,\dots, n.
\end{equation}
This ensures that 
$$
A{\bf u}^{i+1}-A{\bf u}^{i}=0, \ \ \mbox{for}\ \ i=0,1,\dots, n.
$$
Consequently,
\begin{eqnarray*}
(A{\bf u}^{i+1}-A{\bf u}^{i}, \beta ({\bf x}^{i+1}-{\bf x}^i))=0,
\end{eqnarray*}
which implies that
\begin{eqnarray*}
({\bf u}^{i+1}-{\bf u}^{i}, \beta (A^T{\bf x}^{i+1}-A^T{\bf x}^{i}))=0.
\end{eqnarray*}
By using the first equation of \eqref{al:1}, we obtain that 
\begin{eqnarray*}
-({\bf u}^{i+1}-{\bf u}^{i}, \alpha (B^T{\bf y}^{i+1}-B^T{\bf y}^{i}))=0.
\end{eqnarray*}
This yields that
\begin{eqnarray*}
(B{\bf u}^{i+1}-B{\bf u}^{i}, {\bf y}^{i+1}-{\bf y}^{i})=0.
\end{eqnarray*}
Substituting the above equation into the expression of $I_i$ and using the first equation of \eqref{al:1} in the resulting expression, we find that
\begin{eqnarray*}
I_i =-({\bf y}^{i+1}, B {\bf u}^{i+1})=-(B^T{\bf y}^{i+1},  {\bf u}^{i+1})=\frac{\beta}{\alpha}( {\bf x}^{i+1}, A {\bf u}^{i+1}).
\end{eqnarray*}
On the other hand, noticing that the proximity operator ${\cal I}- \mathit {prox}_{\frac{1}{\beta} L}$ is nonexpansive,  we have from \eqref{prox:ineq} and the third equation of \eqref{al:1} that
\begin{eqnarray*}
({\bf x}^{i+1},  {\bf x}^{i+1})\le ({\bf x}^{i+1}, A{\bf u}^{i+1}+{\bf x}^{i+1}), \ \ \mbox{for}\ \ i=0,1,\dots, n.
\end{eqnarray*}
This yields
\begin{equation}\label{a:3}
   ({\bf x}^{i+1}, A{\bf u}^{i+1})\ge 0, \ \ \mbox{for}\ \ i=0,1,\dots, n.
\end{equation}
This yields for $i=0,1,\dots,n$ that $I_i\ge 0$.

In light of \eqref{eq:18} and the non-negativity of $I_i$, we observe that
\begin{eqnarray}\label{eq:18-0}
({\bf y}^{i+1}, {\bf y}^{i+1}-{\bf y}^{i})+(B{\bf u}^{i+1}-B {\bf u}^{i}, B{\bf u}^{i+1})\leq 0,  \ \ \mbox{for}\ \ i=0,1,\dots, n.
\end{eqnarray}
Combining inequality \eqref{eq:18-0} and \eqref{aa:1},  we have hat for $i=0,1,\dots,n,$
\begin{eqnarray*}
\|{\bf y}^{i+1}\|^2+\| {\bf y}^{i+1}-{\bf y}^{i}\|^2-\|{\bf y}^{i}\|^2+\|B{\bf u}^{i+1}\|^2+  \| B {\bf u}^{i}-B{\bf u}^{i+1}\|^2-\|B{\bf u}^{i}\|^2\leq 0.
\end{eqnarray*}
Summing up the above inequalities for $i=0,1, \dots,n$ yields the desired estimate \eqref{a:2}.

To prove \eqref{a00:3},  we define
\[
{\bf a}^{n+1}: = \sum_{i=0}^n\big(\| B {\bf u}^{i}-B{\bf u}^{i+1}\|^2+\| {\bf y}^{i+1}-{\bf y}^{i}\|^2\big).
\]
Clearly, the estimate \eqref{a:2} indicates that 
the sequence $\{ {\bf a}^{n}\}$ is monotonically increasing and has a  upper bound.  Therefore, $\{ {\bf a}^{n}\}$ converges and thus
\[
\lim_{n\rightarrow +\infty} ({\bf a}^{n+1}- {\bf a}^{n})= \lim_{n\rightarrow +\infty} (\| B {\bf u}^{n}-B{\bf u}^{n+1}\|^2+\| {\bf y}^{n+1}-{\bf y}^{n}\|^2)=0.
\]
Then the desired result \eqref{a00:3}  follows. 
 \end{proof}

We are ready to present the convergence result of the iterative algorithm \eqref{al:1}.

\begin{theorem}
The sequence $\{{\bf u}^n\}$ generated by iterative scheme \eqref{al:1} converges to the minimizer of \eqref{2.5}.
\end{theorem}

\begin{proof}
We show that the sequence 
$({\bf y}^{n},{\bf u}^{n}, {\bf x} ^{n})$ generated by \eqref{al:1} converges, and then prove that the limit of the sequence $\{{\bf u}^{n}\}_{n=1}^{\infty}$ is actually the minimizer of \eqref{2.5}. 

Let $S, {\bf v}^n$ be defined in \eqref{M1} and \eqref{vv}, respectively. By using the equation $A^T{\bf x} ^{n}=-\frac{\alpha}{\beta}B^T{\bf y} ^{n}$ and the boundedness of the sequence $\{B{\bf u}^{n}\}$,  $\{{\bf y} ^{n}\}$ in \eqref{a:2},  we conclude that $\{S{\bf v}^{n}\}$ is bounded and thus there exists a subsequence $\{S{\bf v}^{n_k}\}$ 
 of $\{S{\bf v}^{n}\}$  that converges to a point ${\bf a}^*$.  Since $S$ is non-singular, there exists a subsequence $\{{\bf v}^{n_k}\}$ of 
 $\{{\bf v}^{n}\}$ that converges to ${\bf v}^{*}=S^{-1}{\bf a}^*:=({\bf y}^*,{\bf u}^*, {\bf x}^*)^T.$
 Then it follows from  \eqref{al:1}  and \eqref{a00:3} that the triple $({\bf y}^*,{\bf u}^*, {\bf x}^*)^T$ satisfies the equation \eqref{2.90}, which yields, together with  the equation \eqref{Au:1}, that 
 \begin{eqnarray}\label{2.91}
\left\{
\begin{array}{ll}
    \beta A^T{\bf x}^*+ \alpha B^T{\bf y}^*=0, &\\
    {\bf y}^*=({\cal I}-{\rm prox}_{\frac{1}{\alpha}\varphi})(B{\bf u}^*+{\bf y}^*), &\\
    A{\bf u}^*={\bf f}. &
\end{array}\right.
\end{eqnarray}

We next prove that  $({\bf y} ^{n},{\bf u}^{n}, {\bf x} ^{n})$ converges to $( {\bf y} ^{*},{\bf u}^{*}, {\bf x} ^{*})$ as $n\to \infty$.
To this end, we define 
${\bf \bar y}^{n}:={\bf y}^{n}-{\bf y}^{*}$, ${\bf \bar u}^{n}:={\bf u}^{n}-{\bf u}^{*}$ and ${\bf \bar x}^{n}:={\bf x}^{n}-{\bf x}^{*}$,
and show that the sequence $\|{\bf \bar y}^{n}\|^2+
\|B{\bf \bar u}^{n}\|^2$ is nonincreasing.
It follows from the second equation of \eqref{al:1} and that of \eqref{2.91} that 
$$
{\bf y}^{n+1}+B{\bf u}^{n}-B{\bf u}^{n+1}-{\bf y}^*
    ={\cal P}(B{\bf u}^n+{\bf y}^n)-{\cal P}(B{\bf u}^*+{\bf y}^*).
$$
where ${\cal P}:={\cal I}-{\rm prox}_{\frac{1}{\alpha}\varphi}$. Noticing that ${\cal P}$ is firmly nonexpansive, we obtain from the last equation and the definition of firm nonexpansiveness that
\begin{equation}\label{new:1}
     \|{\bf \bar y}^{n+1}+B{\bf u}^{n}-B {\bf u}^{n+1}\|^2\le ({\bf \bar y}^{n+1}+B{\bf u}^{n}-B {\bf u}^{n+1},{\bf \bar y}^{n}+B{\bf \bar u}^{n}).
\end{equation}
In light of the first  equation of \eqref{al:1} and that of \eqref{2.91}, we have for all  $i,j>0$ that
\[
(B{\bf \bar u}^i,{\bf \bar y}^{j})=({\bf u}^{i}- {\bf u}^{*},B^T{\bf \bar y}^{j})=-\frac{\beta}{\alpha}({\bf u}^{i}- {\bf u}^{*},A^T{\bf \bar x}^{j})
=-\frac{\beta}{\alpha}(A{\bf u}^{i}-A {\bf u}^{*},{\bf \bar x}^{j}).
\]   
By using equation \eqref{Au:2} and the third equation of \eqref{2.91}, we have that 
$$
A{\bf u}^{i}-A {\bf u}^{*}=f-f=0.
$$
Combining the equations above yields
\begin{equation}\label{BBBB}
    (B{\bf \bar u}^i,{\bf \bar y}^{j})=0.
\end{equation}
Following the same argument, we obtain that
\begin{equation}\label{BBBB2}
    (B{\bf u}^{i+1}-B{\bf u}^{i},{\bf \bar y}^{j})=0. 
\end{equation}  
Expanding the left side of inequality \eqref{new:1} and using equations \eqref{BBBB} and \eqref{BBBB2}, we get that
\begin{equation}\label{ex:00}
    \|{\bf \bar y}^{n+1}\|^2+\|B{\bf u}^{n}-B {\bf u}^{n+1}\|^2\le ({\bf \bar y}^{n+1},  {\bf \bar y}^{n})+(B{\bf u}^{n}-B {\bf u}^{n+1},  B{\bf \bar u}^{n}). 
\end{equation}
Likewise, by the second equation of \eqref{al:1} and that of \eqref{2.91}, we have that 
\[
B{\bf u}^{n+1}+{\bf y}^n-{\bf y}^{n+1}-B{\bf u}^{*}
    ={\cal P}_0(B{\bf u}^n+{\bf y}^n)-{\cal P}_0(B{\bf u}^*+{\bf y}^*).
\]
where ${\cal P}_0:={\rm prox}_{\frac{1}{\alpha}\varphi}$. 
Using the fact that ${\cal P}_0$ is firmly nonexpansive yields that 
\[
\|B{\bf \bar u}^{n+1}+{\bf y}^n-{\bf y}^{n+1}\|^2\le (B{\bf\bar u}^{n+1}+{\bf y}^n-{\bf y}^{n+1},{\bf \bar y}^{n}+B{\bf\bar  u}^{n}). 
\]
We expand the left and right sides of the above inequality  and use equations \eqref{BBBB} and \eqref{BBBB2} to obtain 
\begin{equation}\label{ex:01}
\|B{\bf \bar u}^{n+1}\|^2+
\|{\bf y}^{n+1}-{\bf y}^{n}\|^2\le ({\bf y}^n-{\bf y}^{n+1},{\bf\bar  y}^{n})+(B{\bf \bar u}^{n+1},B{\bf \bar u}^{n}).
\end{equation}
Combining inequalities \eqref{ex:00} and \eqref{ex:01} leads to 
\[
   \|{\bf \bar y}^{n+1}\|^2+\|B{\bf u}^{n}-B {\bf u}^{n+1}\|^2+ \|B{\bf \bar u}^{n+1}\|^2+
    \|{\bf y}^{n+1}-{\bf y}^{n}\|^2\le \|{\bf \bar y}^{n}\|^2+\|B{\bf \bar u}^{n}\|^2,
\]
which ensures that 
\[
\|{\bf \bar y}^{n+1}\|^2+\|B{\bf\bar u}^{n+1}\|^2\le \|{\bf \bar y}^{n}\|^2+\|B{\bf \bar u}^{n}\|^2. 
\] 
That is, the sequence $\|{\bf \bar y}^{n}\|^2+\|B{\bf \bar u}^{n}\|^2$ is nonincreasing. Moreover, it has a lower bound $0$. Consequently,  there exists a nonnegative number $c$ such that 
\[
\lim_{n\rightarrow +\infty} \|{\bf\bar y}^{n}\|^2+\|B{\bf\bar u}^{n}\|^2=c. 
\]  
Since we already have
 \[
    \lim_{k\rightarrow +\infty} \|{\bf\bar y}^{n_k}\|^2+\|B{\bf\bar u}^{n_k}\|^2=0, 
\]   
we must have $c=0$. Therefore,  the sequences $\{B {\bf u}^{n}\}_{n=1}^{\infty}$ and $\{ {\bf y}^{n}\}_{n=1}^{\infty}$  converge to 
$B {\bf u}^{*}$ and ${\bf y}^*$, respectively. 
  
Furthermore,   as a direct consequence of the first equation of \eqref{al:1} and that of \eqref{2.91},  we have 
\[
    \lim_{n\rightarrow +\infty}   A^T{\bf x}^{n+1}=-\frac{\alpha}{\beta}\lim_{n\rightarrow +\infty}  B^T{\bf y}^{n}=-\frac{\alpha}{\beta}B^T{\bf y}^{*}=A^T{\bf x}^{*}.
\]
Then  the sequence $\{A^T {\bf x}^{n}\}_{n=1}^{\infty}$ converges to $A{\bf x}^*$. Consequently, 
\[
 \lim_{n\rightarrow +\infty} S{\bf v}^{n}=S{\bf v}^{*}. 
\]
Since $S$ is non-singular, the above equation indicates ${\bf v}^{n}-{\bf v}^{*}\rightarrow 0$ as $n\to \infty$.  
Thus,
\[
    \lim_{n\rightarrow +\infty}({\bf y}^{n},{\bf u}^{n},{\bf x}^{n})=({\bf y}^{*},{\bf u}^{*},{\bf x}^{*}).
\]

Finally, as the limits $( {\bf y}^*,{\bf u}^*, {\bf x}^*)$ of $({\bf y} ^{n},{\bf u}^{n}\,{\bf x} ^{n})$ are solutions of \eqref{2.90}, Proposition \ref{lemma:10} ensures that the sequence $\{{\bf u^n}\}_{n=1}^{\infty}$ generated by the iterative scheme \eqref{al:1} converges to the  exact solution of  \eqref{2.5}. This completes the proof.
\end{proof}

\begin{remark}
  As we may observe, the iterative scheme \eqref{al:1} is implicit and the algorithm is convergent for any positive
  constant $\alpha,\beta$.   This implicit algorithm is different from the explicit schemes
  introduced in \cite{li-shen-xu-zhang}, where
  the parameters  $\alpha,\beta$  should be carefully chosen to ensure
  convergence of the algorithm.
  \end{remark}

\section{Computing the Proximity Operator of ${\varphi}$}

In this section, we discuss the calculation of the proximity operator of ${\varphi}$ defined in \eqref{varphi}.

We first consider the case $k=0$ in \eqref{varphi},  i.e.,  the piecewise constant approximation. Note that in this case, ${\varphi}({\bf q})$ is actually the $l_1$ norm of the vector ${\bf q}$, and thus the proximity of ${\varphi}$ can be calculated component-wise as (see, e.g., \cite{Micchelli-Shen-Xu})
\begin{equation}\label{app:1}
   {\rm prox}_{\alpha^{-1} \varphi}({\bf q})=({\rm prox}_{\alpha^{-1} \varphi}( q_1),\ldots, {\rm prox}_{\alpha^{-1} \varphi}( q_{M})),\ \ \mbox{for all}\ \ {\bf q}:=(q_1,\ldots, q_M)^T
\end{equation}
with
\begin{equation}\label{app:2}
    {\rm prox}_{\alpha^{-1} \varphi}( q_i)=\max\{|q_i|-\alpha^{-1},0\}{\rm sign}(q_i).
\end{equation}

We next consider the case $k=1$ in \eqref{varphi},  i.e.,  the piecewise linear approximation. To this end, we first discuss the relationship between the proximity operator of  $\varphi$ and the subdifferential of $\varphi$ at the zero vector. Recalling the definition of $\varphi$ in  \eqref{varphi}, we conclude that the convex function $\varphi$ is positive homogeneous, i.e., for any positive constant $\alpha$, there holds
\[
\varphi(\alpha {\bf q}) = \alpha \varphi( {\bf q}),\ \  \mbox{for all}\ \ {\bf q}.
\]
By $\partial\varphi(\bf 0)$ we denote the subdifferential of $\varphi$  at the zero vector, and by $L_{\partial\varphi(\bf 0)}$ we denote the indicator function of $\partial\varphi(\bf 0)$, Thus, we have that
\begin{eqnarray}\label{LL}
   L_{\partial\varphi(\bf 0)}({\bf v}) &=&\left\{
   \begin{array}{ll}
   0, & \ {\rm if}\ {{\bf v}}\in \partial\varphi(\bf 0), \vspace{1mm}  \\
   +\infty, & \ {\rm otherwise}.
   \end{array}\right.
\end{eqnarray}
For a homogeneous function $\varphi$, we have that
\begin{equation}\label{prox:00}
      ( {\cal I}-{\rm prox}_{ \varphi})({\bf v})={\rm prox}_{L_\partial\varphi(\bf 0)}({\bf v}),\ \ \mbox{for all}\ \ {\bf v}\in \bR^{2(d+1)N_e},
\end{equation}
see, e.g., \cite{Bauschke-Combettes}, Theorem 14.3. Therefore, to compute the proximity operator of ${\varphi}$, we may instead calculate the subdifferential $\partial\varphi(\bf 0)$, since the proximity operator of $L$ can be computed explicitly.

For any ${\bf y}:=(y_1,\ldots, y_{(d+1)N_e})^T\in \partial\varphi({\bf 0})$ with $y_i:=(y_{i,1}, y_{i,2})$,  we have that, from
the definition of the subdifferential and the fact that $\varphi({\bf 0})=0$, 
\[
({\bf y}, {\bf v})\le  \varphi({\bf v}),\ \ \mbox{for all}\ \  {\bf v}:=(v_1,\ldots, v_{(d+1)N_e})^T, \ v_i:=(v_{i,1}, v_{i,2}).
\]
In other words, to  calculate the subdifferential $\partial\varphi(\bf 0)$, we need to find  ${\bf y}$ satisfying the above inequality for all ${\bf v}$. By choosing ${\bf v}=(0,\ldots,v_i,0,\ldots,0)^T, 1\le i\le(d+1)N_e$ in the above inequality and using the definition of $\varphi$  in \eqref{varphi}, we obtain that
\begin{equation}\label{eq:0001}
     y_{i,1}v_{i,1}+y_{i,2}v_{i,2}\le   
     \int_{0}^{1} |v_{i,1}+v_{i,2}s|ds,\ \ \mbox{for all}\ \ (v_{i,1}, v_{i,2})\in\bR^2. 
\end{equation}
We need to solve inequality \eqref{eq:0001} for $y_{i,1}$ and $y_{i,2}$. We first compute the integral on the right-hand side 
of \eqref{eq:0001}.

\begin{lemma}\label{Lemma 8.1}
Let $a, b$ be two real numbers, and
$$
I(a,b):=\int_0^1|a+bs|ds.
$$
(1) If $ab\ge 0$, then 
$$
I(a,b)=|a|+\frac{|b|}{2}.
$$
(2) If $ab<0$ and $b>0$, then 
\begin{eqnarray}\label{new:8.5}
  I(a,b)=\left\{
  \begin{array}{ll}
     -a-\frac{b}{2},  &\ {\rm if}\  a+b\le 0,\\\
    a+\frac{b}{2}+\frac{a^2}{b},  &\ {\rm if}\ a+b> 0.
   \end{array}\right.
\end{eqnarray}
(3) If $ab<0$ and $b<0$, then
\begin{eqnarray}
  I(a,b)=\left\{
  \begin{array}{ll}
     a+\frac{b}{2},  &\ {\rm if}\  a+b\ge 0,\\\
    -a-\frac{b}{2}-\frac{a^2}{b},  &\ {\rm if}\ a+b< 0.
  \end{array}\right.
\end{eqnarray}
\end{lemma}
\begin{proof} Case (1) may be proved by a direct calculation.

We next consider case (2). In case (2), when $a+b\le 0$, we have that
$$
a+bs\leq a+b\leq 0, \ \ \mbox{for}\ \ 0\leq s\leq 1,
$$
since $b>0$. Thus, we obtain that 
$$
I(a,b)=-\int_0^1(a+bs)ds=-a-\frac{1}{2}b. 
$$
When $a+b>0$, we have that $0<-\frac{a}{b}<1$ and thus, for $0\leq s\leq -\frac{a}{b}$, $a+bs\leq 0$ and for $-\frac{a}{b}<s\leq 1$, $0<a+bs$. It follows that
$$
I(a,b)=\int_0^{-\frac{a}{b}}-(a+bs)ds+\int_{-\frac{a}{b}}^1(a+bs)ds=a+\frac{1}{2}b+\frac{a^2}{b}.
$$

In case (3),  we note that 
\[
   \int_0^1|a+bs|ds=\int_0^1|a'+b's|ds
\]
with $a':=-a<0$ and $b':=-b>0$. Then the conclusion in case (2) is valid with $a,b$ replaced by $a',b'$, which yields the result for case (3). 
\end{proof}


We next establish the main result of this section regarding the proximity operator of the convex function $\varphi$ defined by \eqref{varphi} with $k=1$.  To this end, we define 
\begin{equation}\label{omega0}
   \Omega_0:=\left\{(x,y)\in \bR^2: \frac{(1+x)^2}{4}-\frac 12\le y\le \frac 12 -\frac{(1-x)^2}{4},\ |x|\le 1\right\}. 
\end{equation}
Note that set $\Omega_0$ is nonempty because when $|x|\le 1$, 
\[
    \frac 12 -\frac{(1-x)^2}{4}-\left[\frac{(1+x)^2}{4}-\frac 12\right]=1-\frac{(1-x)^2+(1+x)^2}{4}\ge 0.
\]

\begin{proposition}
If $\varphi$ is the convex function defined by \eqref{varphi} with $k=1$, 
then for any  vector ${\bf v}:=(v_1,\ldots,v_{(d+1)N_e})^{T}$ with $v_i:=(v_{i,1},v_{i,2})$,
\begin{equation}\label{eq:112}
prox_{ \varphi}({\bf v})={\bf x}:=(x_1,\ldots, x_{(d+1)N_e})^T,
\end{equation}
where for each $i=1, 2, \dots, (d+1)N_e$,  $x_i:=(x_{i,1},x_{i,2})$ is given by
\begin{eqnarray}\label{eq:113}
x_i=\left\{
   \begin{array}{ll}
    (0,0), &  {\rm if}\ (v_{i,1},v_{i,2})\in\Omega_0 \vspace{1mm}  \\
     (v_{i,1},v_{i,2})-\argmin\limits_{(x,y)\in\Omega_0} (v_{i,1}-x)^2+(v_{i,2}-y)^2, &   {\rm otherwise},
                  \end{array}\right.
\end{eqnarray} 
with $\Omega_0$ being defined by \eqref{omega0}. 
\end{proposition}
\begin{proof}
In light of \eqref{prox:00}, we proceed our proof in three steps: First, we calculate $\partial\varphi({\bf 0})$, the subdifferential of $\varphi$ at the  zero  vector and then we compute the proximity operator of the indicate function $L_{\partial\varphi({\bf 0})}$ according to \eqref{LL}, and finally we obtain the proximity operator of $\varphi$ by using \eqref{prox:00} and the proximity operator of $L_{\partial\varphi({\bf 0})}$.

For any ${\bf y}:=(y_1,\ldots, y_{(d+1)N_e})^T\in \partial\varphi({\bf 0})$, with $y_i:=(y_{i,1}, y_{i,2})$, each component-wise element $y_i$ satisfies the inequality \eqref{eq:0001}. In other words, the calculation of the subdifferential of $\varphi$ at the  zero  vector is to find all $y_i$ satisfying \eqref{eq:0001} for all $(v_{i,1},v_{i,2})\in {\mathbb R}^2$. The right-hand-side of inequality \eqref{eq:0001} is identified as $I(v_{i,1},v_{i,2})$, which may be computed according to Lemma \ref{Lemma 8.1}. Following Lemma \ref{Lemma 8.1} we consider three cases.

{\bf Case 1: $v_{i,1}v_{i,2}\ge 0$.} In this case, from Lemma \ref{Lemma 8.1} we have that $I(v_{i,1},v_{i,2})= |v_{i,1}|+\frac{1}{2}|v_{i,2}|$. It follows from \eqref{eq:0001} that 
\[
v_{i,1}y_{i,1}+v_{i,2}y_{i,2}\le  |v_{i,1}|+\frac{1}{2}|v_{i,2}|, \  \mbox{for all} \ (v_{i,1}, v_{i,2})\in \bR^2 \ \mbox{with}\ v_{i,1}v_{i,2}\ge 0. 
\]
Solving the above inequality we obtain that 
\begin{equation}\label{001}
    |y_{i,1}|\le 1,\ \  |y_{i,2}|\le 1/2.
\end{equation}

{ \bf Case 2: $v_{i,1}<0, v_{i,2}>0$. } In this case, by using \eqref{eq:0001} and taking  $a=v_{i,1},b=v_{i,2}$ in \eqref{new:8.5} of Lemma \ref{Lemma 8.1}, we have that 
\begin{eqnarray}\label{8.10}
  v_{i,1}y_{i,1}+v_{i,2}y_{i,2} \le\left\{
  \begin{array}{ll}
     -v_{i,1}-\frac{v_{i,2}}{2},  &\ {\rm if}\  v_{i,1}+v_{i,2}\le 0,\\\
    v_{i,1}+\frac{v_{i,2}}{2}+\frac{(v_{i,1})^2}{v_{i,2}},  &\ {\rm if}\ v_{i,1}+v_{i,2}> 0.
   \end{array}\right.
\end{eqnarray}
We solve  $y_{i,1}$ and $y_{i,2}$ from the above inequality for all $v_{i,1}<0, v_{i,2}>0$   in  two cases: $v_{i,1}+v_{i,2}\le 0$ and $v_{i,1}+v_{i,2}> 0$. 
    
When $v_{i,1}+v_{i,2}\le 0$, the right-hand-side 
 of \eqref{8.10} is $-v_{i,1}-\frac{v_{i,2}}{2}$, which gives the following inequality
\[
   v_{i,1}(y_{i,1}+1)+v_{i,2}(y_{i,2}+\frac{1}{2})\le 0,\ \  {\rm for\ all}\ \ 0< v_{i,2}\le -v_{i,1}. 
\]
This is equivalent to 
\[
\sup_{0<v_{i,2}\le -v_{i,1} }\left(v_{i,1}(y_{i,1}+1)+v_{i,2}(y_{i,2}+\frac{1}{2})\right)\le 0. 
 \]
We first solve the above inequality for a fixed $v_{i,1}$. To this end, we  define a linear function
\[
g_1(x):=(y_{i,2}+\frac{1}{2})x+v_{i,1}(y_{i,1}+1),\ \ x\in [0,-v_{i,1}].
\]
Then the last inequality implies that 
\[
\sup_{0<x\le -v_{i,1} } g_1(x) \le 0. 
\]  
Now we consider the maximal value of the function $g_1$ over the interval $[0,-v_{i,1}].$ When $y_{i,2}+\frac{1}{2}\ge 0$,  
$g_1$ is monotone increasing  and thus, it has a maximal value at the point $x=-v_{i,1}$. Consequently, we have that  $g_1(-v_{i,1})\le 0$, that is,
\[
   v_{i,1}(y_{i,1}+1)-v_{i,1}(y_{i,2}+\frac{1}{2})\le 0,\ \ {\rm for\ all}\ v_{i,1}<0. 
\]
Solving this inequality, we get that
\begin{equation} \label{yy01}
   y_{i,1}+\frac{1}{2}-y_{i,2} \ge  0,\ \ y_{i,2}+\frac 12\ge 0.
\end{equation}
When $y_{i,2}+\frac{1}{2}< 0$, $g_1$ is monotone decreasing and thus, it has a maximal value at  $x=0$. Therefore,  
\[
    \sup_{0<x\le -v_{i,1} } g_1(x)=g_1(0)\le 0. 
\]
This gives rise to  
\[
v_{i,1}(y_{i,1}+1)\le 0,\ \ \mbox{for all}\ \ v_{i,1}<0.
\]
Solving the above inequality we get that 
 \begin{equation}\label{yy001}
     y_{i,1}+1\ge 0,\ \ y_{i,2}+\frac 12< 0. 
 \end{equation}
Combining \eqref{yy01} and \eqref{yy001}, we conclude that when $v_{i,1}+v_{i,2}\le 0$, the vectors $(y_{i,1}, y_{i,2})$ satisfying inequality  \eqref{8.10} have the form
\begin{eqnarray}\label{yy06}
   y_{i,1}\ge\left\{
  \begin{array}{ll}
     -1,  &\ {\rm if}\    y_{i,2}< -\frac 12,\\
    y_{i,2}-\frac 12,  &\ {\rm if}\  y_{i,2}\ge  -\frac 12 
   \end{array}\right. 
\end{eqnarray} 
  
When $v_{i,1}+v_{i,2}>0$,  we have from \eqref{8.10} that 
\[
   v_{i,1}y_{i,1}+v_{i,2}y_{i,2}\le 
   v_{i,1}+\frac{v_{i,2}}{2}+\frac{v_{i,1}^2}{v_{i,2}}.
\]
This yields the inequality
\begin{equation}\label{eqq:00}
v_{i,2}^2(\frac{1}{2}-y_{i,2})+v_{i,1}v_{i,2}(1-y_{i,1})+v_{i,1}^2\ge 0, \ \  \mbox{for all}\ \ -v_{i,2}< v_{i,1}< 0.
\end{equation}
Define 
\[
g_2(x):=x^2+v_{i,2}(1-y_{i,1})x+v_{i,2}^2(\frac{1}{2}-y_{i,2}),\ \ x\in [-v_{i,2},0].
\]
Hence, solving \eqref{eqq:00} is equivalent to solving
\begin{equation}\label{yy02}
    \inf_{-v_{i,2}< x<0}g_2(x)\ge 0.
\end{equation}
Clearly, the quadratic curve determined by the function $g_2$ has its minimizer at $x^*:=v_{i,2}(y_{i,1}-1)/2$. If $x^*\in [-v_{i,2},0]$, then
\eqref{yy02} is equivalent to
\begin{equation}\label{yy02-0}
    \inf_{-v_{i,2}< x<0}g_2(x)=g_2(v_{i,2}(y_{i,1}-1)/2)\ge 0.
\end{equation}
Solving  \eqref{yy02-0}, we obtain that
\begin{equation}\label{yy07}
\frac{1}{2}-y_{i,2}-\frac{1}{4}(1-y_{i,1})^2\ge 0,\ \ 
-1\le y_{i,1}\le 1.
\end{equation}
If $x^*\notin [-v_{i,2},0]$, then either $x^*< -v_{i,2}$, which is equivalent to $y_{i,1}< -1$, or $x^*>0$, which is equivalent to  $y_{i,1}> 1$.
In the first case, $g_2$ is monotone increasing over $[-v_{i,2},0]$ and thus 
\[
   \inf_{-v_{i,2}< x<0}g_2(x)=g_2(-v_{i,2})\ge 0.
\]
Solving the above inequality we obtain that
\begin{equation}\label{yy04}
    y_{i,1}+\frac 12-y_{i,2}\ge 0,\ \ y_{i,1}< -1. 
\end{equation}
In the second case, $g_2$ is monotone decreasing over $[-v_{i,2},0]$ and thus 
\[
\inf_{-v_{i,2}< x<0}g_2(x)=g_2(0)\ge 0,
\]
which yields 
\begin{equation}\label{yy05}
\frac 12-y_{i,2}\ge 0,\ \ y_{i,1}\ge 1. 
\end{equation}
   
Combining \eqref{yy07}, \eqref{yy04} and \eqref{yy05}, we get that when $v_{i,1}+v_{i,2}>0$,
\begin{eqnarray}\label {yy03}
   y_{i,2}\le\left\{
  \begin{array}{ll}
     \frac{1}{2}-\frac{1}{4}(1-y_{i,1})^2,  &{\rm if}\  -1\le y_{i,1}\le 1,\\
    y_{i,1}+\frac 12,  &\ {\rm if}\  y_{i,1}<  -1,\\
    \frac 12,  &\ {\rm if}\  y_{i,1}>  1
   \end{array}\right. 
\end{eqnarray}

 { \bf Case 3: $v_{i,2}<0, v_{i,1}>0$.}
In this case, note that   $-v_{i,1}<0, -v_{i,2}>0$, and \eqref{eq:0001} can be rewritten as
 \[
    -v_{i,1}(-y_{i,1})-v_{i,2}(-y_{i,2})\le I(v_{i,1},v_{i,2})=I(-v_{i,1},-v_{i,2}).
 \]
By employing the conclusion in Case 2, we obtain that  
\begin{eqnarray}\label{yy08}
   -y_{i,1}\ge\left\{
  \begin{array}{ll}
     -1,  &\ {\rm if}\    y_{i,2}> \frac 12,\\
    -y_{i,2}-\frac 12,  &\ {\rm if}\  y_{i,2}\le \frac 12,  
   \end{array}\right. 
\end{eqnarray} 
and 
\begin{eqnarray}\label{002}
   -y_{i,2}\le\left\{
  \begin{array}{ll}
     \frac{1}{2}-\frac{1}{4}(1+y_{i,1})^2,  &\ {\rm if}\  -1\le y_{i,1}\le 1,\\
    -y_{i,1}+\frac 12,  &\ {\rm if}\  y_{i,1}> 1,\\
    \frac 12,  &\ {\rm if}\  y_{i,1}< -1. 
   \end{array}\right. 
 \end{eqnarray}

Note that \eqref{eq:0001} is valid for all $v_{i,1},v_{i,2}$. This requires that $y_{i,1}$, $y_{i,2}$ must satisfy all of \eqref{001}, \eqref{yy06}, \eqref{yy03},  \eqref{yy08} and \eqref{002}. Thus, we conclude that  
\begin{equation*}
-\frac 12 +\frac{1}{4}(1+y_{i,1})^2\le y_{i,2}\le \frac 12 -\frac{1}{4}(1-y_{i,1})^2,\ \  |y_{i,1}|\le 1.
\end{equation*}
Consequently, the subdifferential of $\varphi$ at the  zero  vector can be represented as 
\[
\partial\varphi({\bf 0})=\{ {\bf y}:=(y_1,\ldots, y_M)^T:  y_i:=(y_{i,1}, y_{i,2})\in\Omega_0 \}.
\]

We next calculate the proximity operator of the indicator function $L_{\partial\varphi(\bf 0)}$. 
Given a ${\bf v}:=(v_1,\ldots,v_{(d+1)N_e})^{T}$ with $v_i:=(v_{i,1},v_{i,2})$, we let 
${\bf z}:={\rm prox}_{L_\partial\varphi(\bf 0)}({\bf v})$,
where
${\bf z}:=(z_1,\ldots, z_{(d+1)N_e})^T$, with $z_i:=(z_{i,1}, z_{i,2})$.
Recalling the definition of $L_{\partial\varphi(\bf 0)}$ in \eqref{LL} and its proximity operator, we conclude that the proximity operator of the indicator function $L_{\partial\varphi(\bf 0)}$ at ${\bf v}$ is the $L^2$ projection of ${\bf v}$ onto the set $\Omega_0$. This implies that 
\begin{eqnarray*}
z_i&=&\left\{
\begin{array}{ll}
    (v_{i,1},v_{i,2}), &  {\rm if}\ (v_{i,1},v_{i,2})\in\Omega_0 \vspace{1mm}  \\
     \argmin\limits_{(x,y)\in\Omega_0,} \{(v_{i,1}-x)^2+(v_{i,2}-y)^2\}, &   {\rm otherwise}.
\end{array}\right.
\end{eqnarray*}

Finally, by using formula \eqref{prox:00}, we have that
\[
    {\rm prox}_{ \varphi}({\bf v})={\bf v}-{\rm prox}_{L_\partial\varphi(\bf 0)}({\bf v})={\bf v}-{\bf z}.
\]
Therefore,   \eqref{eq:112} and \eqref{eq:113}  follow immediately. 
\end{proof}

Computing the proximity operator of function $\varphi$ with a general order $k\ge 2$ requires further investigation. Below, we describe an approximation of  ${\rm prox}_{\varphi}$ for general $k$.
It was shown in \cite{jing-li-xu} that  for any ${\bf q}:=(q_1,\ldots, q_{(d+1)N_e})^T$ with $ q_i:=(q_{i,1},\ldots,  q_{i,k+1})$,
there exist positive $\gamma,\eta$ such that
\[
\gamma \|{\bf q}\|_* \le {\varphi}({\bf q})\le \eta \|{\bf q}\|_*,\ \  {\rm with}\  \|{\bf q}\|_*:=\sum_{i=1}^{(d+1)N_e}\sum_{m=0}^{k}\frac{1}{m+1} |q_{i,m+1}|.
\]
A re-scaling with $\bar {\bf q}:=(\bar q_1,\ldots,\bar q_{(d+1)N_e})^T$, $\bar q_i:=(\bar q_{i,1},\bar q_{i,2},\ldots, \bar q_{i,k+1})$, $\bar q_{i,j}=\frac{1}{j} q_{i,j}$, for $j=1,2,\dots, k+1$ yields that
\[
\gamma \|{\bar {\bf q}}\|_{l_1} \le {\varphi}({\bf {\bf q}})\le \eta \|{\bar {\bf q}}\|_{l_1}. 
\]
In other words, the function $\varphi$ at the vector ${\bf q}$ is equivalent to the $\ell_1$ norm of ${\bar {\bf q}}$. Noting that the proximity operator of the $\ell_1$ norm can be computed by using \eqref{app:1}-\eqref{app:2}, we may use the proximity operator of the $\ell_1$ norm to approximate the proximity operator of ${\varphi}$.



\section{Numerical Results}
In this section, we present three numerical examples to verify the theoretical findings in the previous sections. We use the primal-dual weak Galerkin method \eqref{min:uh-new} based on the $L^p$ stabilizer
to solve the elliptic equation \eqref{poission} in two-dimensional setting for $p=1,2,\infty$.  Without loss of generality, we test $k=2$ in our numerical experiments. That is, the
  FE space $V_h$ in \eqref{Vh} is given by
 \[
   V_h:=\{v:=\{v_0,v_b, {\bf v}_g\}: \{v_0,v_b, {\bf v}_g\}|_{T}\in P_2(T)\times P_2(e)\times [P_{1}(e)]^2,\ e\in \partial T, T\in {\cal T}_h \}.
\]
The dual space $W_h$ for the Lagrange multiplier is chosen as
\[
      W_h:=\{\sigma: \sigma|_{T}\in P_{1},\ \ \forall T \in{\cal T}_h\}.
\]
We take the domain $\Omega:=[0,1]\times [0,1]$, and obtain our the triangular partitions by successively applying a uniform refinement procedure that
divides each coarse level thriangel into four congruent sub-triangles by connecting the tree mid-points on the edges of each triangle.
In our numerical experiment, we test various errors including  the discrete $W^{2,p}$-norm (defined in \eqref{discrete-H2}), $W^{1,p}$- and
$L^p$-norms.

  \subsection{Numerical results for continuous constant coefficients}
    We first consider the problem \eqref{poission}  with continuous constant coefficients. That is,  the coefficients $a_{ij}$  is taken as  following:
  \[
     a_{11}=1,\ \ a_{12}=a_{21}=1,\ \ a_{22}=6.
  \]
   The right-hand side function is chosen such that the exact solution is
  \[
     u(x,y)=\sin(\pi x)\sin(\pi y).
  \]
    List in Table \ref{1} are the approximation errors and   convergence rates for the WG solution  based on $L^p$ stabilizer with $p=1,2,\infty$.
    As we may observe, the convergence rate of $\|u-u_h\|_{0,p}$ and $\|u-u_h\|_{1,p}$ is separately $k+1$ and $k$ for all $p=1,2,\infty$, which are consistent with
    the results \eqref{es:2} and \eqref{est:3} in Theorem \ref{theo:1} and Theorem \ref{theo:2}.  As for the discrete $W^{2,p}$ norm, the convergence rate is $k-1$ for both $p=1$ and $p=2$.
    While for $p=\infty$, it seems that the convergence rate can arrive at $k$, $1$ order higher than the error estimate given by \eqref{EQ:Feb25:100}.

  \begin{table}[htbp]\caption{Various errors and  convergence rates based on the $L^p$ stabilizer for continuous constant coefficients.}\label{1} \centering
\begin{threeparttable}
        \begin{tabular}{c |c c c c c c c }
        \hline
    & N &   $\|u-u_h\|_{2,p,h}$ & order& $\|u-u_h\|_{1,p}$  & order & $\|u-u_h\|_{0,p}$ & order   \\
       \hline \cline{2-8}
& 4&  5.64e-00 &  -- &  2.56e-01  & -- & 1.17e-02  & --\\
 & 8&   2.79e-00 &  1.04&  4.70e-02 &  2.45&  1.08e-03  & 3.44 \\
$p=1$ & 16&  1.42e-00 & 0.99  & 9.87e-03   &2.25 & 1.07e-04 &  3.33 \\
& 32&   7.19e-01 &  0.98&   2.34e-03  & 2.07  & 1.24e-05 &  3.11 \\
& 64&   3.62e-01 &  0.99&   5.77e-04  & 2.02  & 1.51e-06 &  3.04 \\
 \hline \cline{2-8}
               & 4&    1.45e-00  & -- &  1.05e-00 & -- &  4.60e-02&    -- \\
                & 8&     5.69e-01 &  1.35 &  1.80e-01 &  2.54&   4.20e-03 & 3.45 \\
$p=\infty$ & 16&    1.65e-01  & 1.79 &  4.19e-02 &  2.10  &4.03e-04   &3.38 \\
               & 32&     4.31e-02 &  1.94&  1.06e-02  & 1.99 & 4.99e-05   &3.01 \\
                & 64&     1.09e-02 &  1.98&  2.64e-03  & 2.00 & 6.22e-06 & 3.00\\
 \hline \cline{2-8}
                  & 4&    8.41e-01&   --  & 1.72e-01 &  -- &  1.12e-02  & -- \\
                & 8&    4.99e-01 &  0.75  & 3.42e-02  & 2.34 &  1.09e-03  & 3.36 \\
       $p=2$  & 16& 2.62e-01  & 0.93   &8.18e-03  & 2.06   &1.28e-04 &  3.09 \\
                & 32&    1.34e-01 & 0.97&  2.04e-03  & 2.01& 1.59e-05  & 3.00  \\
                &64&   6.74e-02 & 0.99 & 5.09e-04  & 2.00& 2.00e-06  & 3.00  \\
 \hline
 \end{tabular}
 \end{threeparttable}
\end{table}

 \subsection{Numerical results for continuous variable coefficients}
   We now suppose  coefficients $a_{ij}$ in  \eqref{poission}  are variable functions with
  \[
     a_{11}=1+x,\ \ a_{12}=a_{21}=0.5xy,\ \  a_{22}=1+y.
  \]
  We still choose a right-hand side function $f$ such that the exact solution is
  \[
     u(x,y)=\sin(\pi x)\sin(\pi y).
  \]

  We present in Table \ref{2}  the numerical results for the problem \eqref{poission}  with continuous variable coefficients.
  From Table \ref{2},
    we  can observe similar convergence rates as those for continuous constant coefficients. That is, we see an optimal convergence order of
    $k+1$ and $k$ for the $L^p$ and $W^{1,p}$ norm, respectively for all $p=1,2,\infty$, and an optimal convergence rate $k-1$ for the discrete  $W^{2,p}$ norm
   for $p=1,2$. All the numerical results confirm the theoretical findinds given in \eqref{EQ:Feb25:100}, \eqref{es:2} and \eqref{est:3}. Again, we observe a convergence  order of $k$ for
   the error $\|u-u_h\|_{2,p,h}$ based on the
    $L^{\infty}$ stabilizer,  which is $1$ order higher than the  optimal convergence rate $k-1$.

 \begin{table}[htbp]\caption{Various errors and  convergence rates based on the $L^p$ stabilizer for continuous variable coefficients.}\label{2}\centering
\begin{threeparttable}
        \begin{tabular}{c |c c c c c c c }
        \hline
    & N &   $\|u-u_h\|_{2,p,h}$ & order& $\|u-u_h\|_{1,p}$  & order & $\|u-u_h\|_{0,p}$ & order   \\
       \hline \cline{2-8}
                       & 4  &2.32e-00  & --  &   2.25e-01 &  --  & 1.02e-02 &  --\\
  $p=1$ &  8   &1.02e-00 &  1.19 &  4.87e-02  & 2.21   &1.14e-03 &  3.16\\
                      & 16  & 4.79e-01 &  1.08 &  1.14e-02  & 2.10  & 1.35e-04&   3.07\\
                     &  32  &2.38e-01  & 1.01&   2.78e-03 &  2.04   &1.73e-05  & 2.97\\
                      & 64 &1.19e-01  & 0.99&   6.88e-04 &  2.01   &2.46e-06  & 2.97\\
 \hline \cline{2-8}
               & 4 & 1.90e-00   &--&  7.23e-01 &  --  & 3.76e-02  & -- \\
              & 8 &  5.41e-01 & 1.81 & 1.54e-01  & 2.22  & 3.98e-03  & 3.24  \\
$p=\infty$ & 16 &  1.40e-01   &1.96  & 4.02e-02&   1.95&  4.69e-04  & 3.08  \\
              & 32&   3.44e-02   &2.01  & 1.02e-02  & 1.98  & 5.76e-05  & 3.02\\
              & 64  &   8.27e-03 &  2.06  & 2.55e-03  & 1.99  & 8.16e-06  & 2.82 \\
 \hline \cline{2-8}
               & 4 &8.47e-01  &  -- &  1.68e-01   & --   & 1.19e-02   & -- \\
   $p=2$ & 8  & 4.95e-01  &  0.75  &  3.98e-02   & 2.08   &1.40e-03   & 3.08\\
               & 16&  2.60e-01   & 0.93   &9.85e-03    &2.02 & 1.73e-04   & 3.02\\
                & 32 &   1.33e-01    &0.97   & 2.46e-03  &  2.00 &  2.23e-05   & 2.96\\
                & 64 &   6.70e-02    &0.99   & 6.14e-04  &  2.00 &  3.13e-06   & 2.83\\
 \hline
 \end{tabular}
 \end{threeparttable}
\end{table}

 \subsection{Numerical results for discontinuous  coefficients}
   In this subsection, we consider \eqref{poission} with discontinuous coefficients, which is  given by
  \[
      a_{11}=2,\ \  a_{12}=a_{21}=\frac{x-0.5}{|x-0.5|}\frac{y-0.5}{|y-0.5|},\ \  a_{22}=2.
  \]
    The right-hand side function $f$ is chosen such that the exact solution is
  \[
     u(x,y)=xy(1-e^{1-x}) (1-e^{1-y}).
  \]
   The computational results are given in Table \ref{3}.

  From Table \ref{3} we observe that the error $\|u-u_h\|_{2,p,h}$ and $W^{1,p}$ converges separately with order
  $k-1$ and $k$, for all $p=1,2,\infty$, which are consistent with the results in Theorems \ref{theo:001} and \ref{theo:2}.
  Note that $a_{12}$ and $a_{21}$ in this case  is discontinuous across the line $x=0.5$ and $y=0.5$.
  As expected, we do not see the optimal convergence rate for the  error $\|u-u_h\|_{0,p}$.

 \begin{table}[htbp]\caption{Various errors and  convergence rates based on the $L^p$ stabilizer for discontinuous  coefficients.}\label{3}\centering
\begin{threeparttable}
        \begin{tabular}{c |c c c c c c c }
        \hline
    & N &   $\|u-u_h\|_{2,p,h}$ & order& $\|u-u_h\|_{1,p}$  & order & $\|u-u_h\|_{0,p}$ & order   \\
       \hline \cline{2-8}
              & 4 &5.79e-01 &  --  & 2.63e-02 &  -- & 1.60e-03 &  --\\
      & 8 &   3.29e-01  & 0.81 &  5.45e-03  & 2.27 &  3.08e-04 &  2.37\\
$p=1$    & 16&  1.68e-01  & 0.97   &1.18e-03  & 2.20&  5.36e-05  & 2.52\\
           &32 &   8.54e-02 &  0.99&   2.67e-04  & 2.14 & 8.42e-06  & 2.67\\
            &64 &   4.30e-02 &  0.99&   6.30e-05  & 2.09 & 1.23e-06  & 2.77\\
  \hline \cline{2-8}
                    & 4 & 1.61e-01  & --&  1.03e-01  & --& 4.62e-03  & --\\
                 & 8 & 4.55e-02  & 1.83  & 3.10e-02  & 1.73& 1.05e-03  & 2.13\\
$p=\infty$ & 16 & 1.71e-02   &1.41& 8.52e-03  & 1.87  & 1.93e-04   &2.45\\
                   & 32&   7.49e-03  & 1.19 & 2.23e-03 &  1.94 & 3.16e-05 &  2.60\\
                    & 64&  3.51e-03  & 1.09 &5.68e-04 &  1.97 & 4.84e-06 &  2.70\\
     \hline \cline{2-8}
             & 4&  1.10e-01  & --  & 1.88e-02 &  --  & 1.30e-03 &  --\\
              & 8&  6.65e-02  & 0.73  & 4.31e-03&  2.12 &2.32e-04 &  2.49\\
    $p=2$    & 16& 3.59e-02  & 0.89& 1.01e-03  & 2.09&  4.03e-05   &2.52\\
           & 32&  1.85e-02  & 0.95 &  2.43e-04&   2.06&  6.56e-06 &  2.62\\
            & 64&  9.41e-03  & 0.98 &  5.90e-05&   2.04&  1.05e-06 &  2.64\\
   \hline
 \end{tabular}
 \end{threeparttable}
\end{table}

\vfill\eject

\end{document}